\documentclass[12pt]{amsart}
\usepackage[dvips]{color}
\usepackage{amsmath}
\usepackage{amsxtra}
\usepackage{amscd}
\usepackage{amsthm}
\usepackage{amsfonts}
\usepackage{amssymb}
\usepackage{eucal}
\usepackage{epsfig}
\usepackage{graphics}
\usepackage{accents}
\usepackage{simplewick}
\usepackage{dynkin-diagrams}

\usepackage{mathtools}
\usepackage{tikz}
\usepackage[all]{xy}
\usepackage{here} 
\usetikzlibrary{calc}
\numberwithin{equation}{section}
\newtheorem{thm}{Theorem}[section]
\newtheorem{prop}[thm]{Proposition}
\newtheorem{lem}[thm]{Lemma}
\newtheorem{rem}[thm]{Remark}

\newtheorem{conj}[thm]{Conjecture}

\newcommand{\nn}{\nonumber}

\theoremstyle{definition}

\newcommand{\ket}[1]{{| #1 \rangle}}      %ket
\newcommand{\C}{{\mathbb C}}
\newcommand{\Z}{{\mathbb Z}}

\newcommand{\E}{{\mathcal E}}
\newcommand{\F}{\mathcal F}

\newcommand{\gl}{\mathfrak{gl}}

\newcommand{\End}{\mathop{\rm End}}

\newcommand{\id}{{\rm id}}

\newcommand{\Sym}{\mathrm{Sym}}

\newcommand{\wt}{{\rm wt}\,}

\newcommand{\on}{\operatorname}
\newcommand{\mc}{\mathcal}

\newcommand{\cont}[2]{\contraction[1ex]{}{#1}{}{#2} #1 #2}

\newcommand{\ce}{\check{e}}
\newcommand{\cf}{\check{f}}

\newcommand{\cq}{\check{q}}
\newcommand{\bq}{\bar{q}}

\newcommand{\cs}{\check{s}}

\newcommand{\cb}{\check{\beta}}

\newcommand{\cE}{\check{\E}}
\newcommand{\ck}{\check{k}}

\newcommand{\calA}{\mathcal{A}}
\newcommand{\Fb}{\mathbb F}
\newcommand{\Fbb}{\bar{\mathbb F}}
\newcommand{\Sc}{\mathcal{S}}
\newcommand{\Rv}{\tilde R}  %{{R^\vee}}
\newcommand{\Vir}{{ Vir}}

\newcommand{\jm}{}

\begin{document}

\begin{title}[Extensions of quantum toroidal $\gl_1$]
{Extensions of a commuting pair of quantum toroidal $\gl_1$}
\end{title}

\author{B. Feigin, M. Jimbo, and E. Mukhin}%$^{\ast}$} %\thanks{$^{\ast}$Corresponding author.}
\address{BF: 
National Research University Higher School of Economics,  101000, Myasnitskaya ul. 20, Moscow,  Russia;
Hebrew University of Jerusalem, Einstein Institute of Mathematics,
Givat Ram. Jerusalem, 9190401, Israel
}
\email{borfeigin@gmail.com}

\address{MJ: 
Professor Emeritus,
Rikkyo University, Toshima-ku, Tokyo 171-8501, Japan}
\email{jimbomm1@gmail.com}

\address{EM: Department of Mathematics,
Indiana University Indianapolis,
402 N. Blackford St., LD 270, 
Indianapolis, IN 46202, USA}
\email{emukhin@iu.edu}

%81R10: "Infinite-dimensional groups and algebras motivated by physics, including Virasoro, Kac-Moody, W-algebras and other current algebras and their representations" (MSC2010)
% 81R12: "Relations with integrable systems" (MSC2010)
% 17B69: "Vertex operators; vertex operator algebras and related structures" (MSC2010)
% 17B80: "Applications to integrable systems" (MSC2010)

\begin{abstract} 
We introduce a family of algebras $\mathcal{A}_{M,N}$, $M,N\in\Z$, 
as an extension of a pair of commuting quantum toroidal
$\gl_1$ subalgebras $\E_1,\check{\E}_1$, wherein the parameters are tuned 
in a specific way according to $M,N$.  
In the case $M=\pm 1$, algebra $\mathcal{A}_{\pm1,N}$
is a shifted quantum toroidal $\gl_2$ algebra
introduced in \cite{FJM2}.
Conjecturally there is a coproduct homomorphism
$\mathcal{A}_{M,N_1+N_2}\to\mathcal{A}_{M,N_1}\hat\otimes\mathcal{A}_{M,N_2}$
to a completed tensor product,
whose restriction to  the subalgebras $\E_1,\check{\E}_1$ 
coincides with the standard \jm{Drinfeld coproduct}. %the latter.
We give examples of $\mathcal{A}_{M,N}$ modules constructed
on certain direct sums of tensor products of Fock modules 
of $\E_1\otimes\check{\E}_1$. 
\end{abstract}

\dedicatory{To the memory of Masatoshi Noumi}

\keywords{vertex operators, intertwiners, quantum toroidal algebra}

\hfill  \today

\maketitle

\section{Introduction}

Construction of 
Vertex Operator Algebras (VOAs) is an important and non-trivial task. 
It has a long history and a vast literature exists on this subject.
Various methods of construction have been known 
including the coset construction, free field realization and the
Drinfeld-Sokolov reduction. 

Yet another approach is to extend a given VOA
by adjoining local fields that transform as its representations. 
To illustrate the idea, let us take the Virasoro algebra $\Vir_\beta$
with the central charge $c=13-6(\beta+\beta^{-1})$. 
The primary fields $\varphi_{1,2}(z),\varphi_{2,1}(z)$
provide representations of $\Vir_\beta$, 
but they are not mutually local because their correlation
functions develop non-trivial monodromy. 
A way out of this problem is to
prepare two copies of the Virasoro algebra $\Vir_\beta$, 
${\Vir}_{\check{\beta}}$ and 
impose a relation between their central charges
\begin{align*}
\beta+\check{\beta}=M+1\qquad \in\Z\,.
\end{align*}
Let $\varphi^{\pm}_{1,2}(z)$, $\check{\varphi}^{\pm}_{2,1}(z)$ be
the components of the corresponding primary fields.
Under the condition above, the monodromy cancels out from the combination
\begin{align*}
X(z)= \varphi^+_{1,2}(z)\check{\varphi}_{2,1}^-(z) 
+\varphi^-_{1,2}(z)\check{\varphi}_{2,1}^+(z) \,,
\end{align*}
yielding a local field 
which can be used to obtain a new VOA extending
$\Vir_\beta\oplus \Vir_{\check\beta}$, \jm{see \cite{La}.}
This sort of a ``gluing'' method has been studied  
for a large class of VOAs 
\cite{BBFLT}, \cite{BFL}, \cite{PR}, \cite{CG}, \cite{B}, \cite{LL}.

In this article we consider a $q$-analog of the gluing  construction 
in the setting of quantum toroidal algebras.
We start with a pair of commuting quantum toroidal $\gl_1$ algebras
$\E_1=\E_1(q_1,q_2,q_3)$, $\check{\E}_1=\E_1(\cq_1,\cq_2,\cq_3)$,
and impose the relations 
\begin{align*}
\cq_1=q_1^{-1}\,,\quad \cq_2=q_2q_1^{-M+1}\,,\quad 
\cq_3=q_3q_1^{M+1}\,,\quad \check{C}=C q_1^{N/2}\,.
\end{align*}
Here $M,N\in \Z$, and $C,\check{C}$ are central elements of 
$\E_1,\check{\E}_1$, respectively. 
Let further 
$\E_1[K]$, $\check{\E}_1[\check{K}]$ 
be their extensions by invertible split central elements $K$, $\check{K}$.
We introduce an algebra 
$\mathcal{A}_{M,N}=\mathcal{A}_{M,N}(q_1,q_2,q_3)$ 
as an extension of 
$\E_1[K]\otimes\check{\E}_1[\check{K}]$ by currents
$\{X^+_i(z)\}_{i\in\Z+N/2}$ and $\{X^-_i(z)\}_{i\in\Z}$
which transform as vector representations under the adjoint actions of
$\E_1[K]$
 and $\check{\E}_1[\check{K}]$, see 
(R1) in Section \ref{sec:defAMN}.
Our goal is to write down an explicit presentation of $\mathcal{A}_{M,N}$
by generators and relations, and to give its representations
in terms of free fields. We also conjecture that there is a coproduct
homomorphism 
$\mathcal{A}_{M,N_1+N_2}\to\mathcal{A}_{M,N_1}\hat\otimes\mathcal{A}_{M,N_2}$
to a completed tensor product,
which extends the standard coproduct of the subalgebra
$\E_1[K]\otimes\check{\E}_1[\check K]$. 

In fact, $\mathcal{A}_{M,N}$ is not an algebra in the usual sense, 
as the relations involve fused currents whose Fourier components are
infinite sums of generators. These relations are well-defined   
on admissible representations
in the sense discussed in Section \ref{sec:fused currents}
and Remark \ref{rem:admissible} below. 
Admissible representations considered in this paper are 
direct sums of tensor products of Fock representations of $\E_1,\check{\E}_1$.

The special case $M=1$ was treated in our previous paper \cite{FJM2}. 
There we have shown that,
on admissible representations, $\mathcal{A}_{1,0}$ is essentially 
isomorphic to an extension $\E_2[K]$ of
the quantum toroidal $\gl_2$ algebra $\E_2$ by 
an invertible (but not central) element $K$. 
Due to the symmetry 
$\mathcal{A}_{M,N}(q_1,q_2,q_3)\simeq \mathcal{A}_{-M,N}(q_1,q_3,q_2)$,
$\mathcal{A}_{-1,0}$ also 
is isomorphic to $\E_2[K]$ with different parameters.
For even $N\ge 0$, we recover the extensions of $\gl(N+2|1)$ discussed in \cite{FJM}
as representations of $\mathcal{A}_{-1,N}$, see Example 3 
in Section \ref{sec:mn}.
In addition, $\mathcal{A}_{0,0}$ is related to the quantum toroidal 
$\gl_{1|1}$ algebra $\E_{1|1}$, see Section \ref{sec:N=0}.

In \cite{LL}, a closely related construction 
is discussed for affine Yangian $\gl_1$ with $M\in\{1,0,-1\}$ and general $N$. 
The approaches in \cite{LL} and in the present paper are 
``orthogonal'' to each other: 
the presentation in \cite{LL} is based on the picture of plane partitions, 
whereas we deal here with free fields and vertex operators. 
This seems to suggest that $\mathcal{A}_{M,N}$ allows for 
another set of generators whose action has a combinatorial
description in terms of plane partitions, and
that the two sets of generators are related 
by an action of $SL(2,\Z)$. 
\medskip

We now outline the content of the paper in some detail.

The defining relations of $\mathcal{A}_{M,N}$
are given in (R1)--(R3) in Section \ref{sec:defAMN}.
Relation (R1) says that $X^\pm_{\mp i}(q_1^i z)$ transform 
as vector representations of $\E_1[K]$, 
while $X^\pm_i(z)$ transform as vector representations of 
$\check{\E}_1[\check K]$.
In particular, for any chosen $i,j$, the currents  $X^+_i(z)$, $X^-_j(z)$
together with  $\E_1[K]\otimes\check{\E}_1[\check K]$ generate
$\mathcal{A}_{M,N}$. 
Relation (R2) is a set of quadratic relations among 
$X^\pm_i(z)$ and $X^\pm_j(z)$.
Relation (R3) gives the commutator $[X_i^+(z),X^-_j(w)]$ as a 
sum of fused currents from $\E_1[K]\otimes\check{\E}_1[\check K]$
(when $M$ is even, the commutator should be understood in the sense of superalgebras).
The integer $N$ enters only (R3) and, in particular, 
$[X_i^+(z),X^-_j(w)]=0$ holds if $N/2<i+j<-N/2$. Such relations are
characteristic to the ones in the (anti-dominantly)
shifted quantum affine \cite{FT} or 
quantum toroidal \cite{N},\cite{NW} algebras. 
It should be noted however that the algebra $\mathcal{A}_{M,N}$ is not the same as those
in the literature: the main difference is that
the role of Cartan currents in the latter 
is played by the quantum toroidal $\gl_1$ algebras 
$\E_1[K]\otimes\check\E_1[\check K]$.

We conjecture that there is a ``coproduct'' homomorphism to a completed 
tensor product
\begin{align*}
\Delta_{N_1,N_2}\,&:\,
\calA_{M,N_1+N_2}
\longrightarrow
\calA_{M,N_1}\hat \otimes \calA_{M,N_2}
\,,
\end{align*}
whose restriction to $\E_1[K]\otimes\check\E_1[\check K]$
coincides with the standard coproduct of the latter, \jm{see \eqref{coproE1} below,}
and 
\begin{align*}
&\Delta_{N_1,N_2} X^+_i(z) 
=X^+_i(z)\otimes 1+ \mathcal{R}^{-1}(1\otimes X^+_i(z))\mathcal{R}\,,\\
&\Delta_{N_1,N_2} X^-_i(z) 
=\mathcal{R}^{-1}(X^-_i(z)\otimes 1)\mathcal{R}+ 1\otimes X^-_i(z)\,,
\end{align*}
where $\mathcal{R}$ 
is the universal $R$ matrix of $\E_1\otimes\check{\E}_1$.
See Conjecture \ref{conj:copro}.
The conjecture says in particular that $\mathcal{A}_{M,0}$ is a 
\jm{topological} 
Hopf algebra.

In many cases $\mathcal{A}_{M,N}$ has representations 
given in terms of free fields.
The building blocks are the Fock modules $\F_c(v)$ , $c\in\{1,2,3\}$, 
of $\E_1$ in the free field realization,
and certain intertwiners relating a Fock module and its
tensor product with a vector representation.
The most basic case is $N=M-1$.  
In Theorem \ref{thm:F22} we show that $\mathcal{A}_{M,M-1}$ 
has a representation on the direct sum 
\begin{align*}
\mathbb{F}_{2;2}(v;\check v) 
=\bigoplus_{n\in\Z}\F_2(q_3^{-n}v)\boxtimes \check\F_2(\cq_3^{-n}\check v)\,,
\end{align*}
where the generators $X_i^\pm(z)$ are realized by single vertex operators. 

Actually, in the definition of $\mathcal{A}_{M,N}$,
we choose $q_1$ to play a special role, 
breaking the symmetry between $q_1,q_2,q_3$. 
As a consequence, the Fock module of color $c=1$ has a special feature. 
Namely the $\E_1[K]$ module $\F_1(v)$ becomes 
an $\mathcal{A}_{M,1}$ module  
by letting $\check \E_1[\check K]$ act trivially and $X^\pm_i(z)$ act as $0$. 
A similar statement holds for the $\check\E_1[\check K]$ module $\check\F_1(v)$.

The coproduct conjecture holds true when applied to tensor multiplication by 
an $\mc A_{M,1}$ module
$\F_1(v)$ or $\check\F_1(v)$
with an arbitrary representation, see Theorem \ref{coproduct thm}.
In particular, tensor products with $\mathbb{F}_{2;2}(v;\check v)$
yield the following. 
Setting 
\begin{align*}
&\F_{c_1,\ldots,c_l}(\mathbf{v})
=\F_{c_1}(v_{1})\otimes\cdots\otimes\F_{c_l}(v_{l})\,,
\quad c_j\in\{1,2,3\},
\\
&\mathbf{v}p^{\delta_a}=(v_1,\ldots,\overset{\underset{\smile}{a}}{v_a}p,
\ldots,v_l)\,,
\end{align*}
and similarly for the checked version, we have a representation of
 $\mathcal{A}_{M,N}$ with $N=M+m+\check m-3$ for arbitrary $m,\check m\ge1$
on the space
\begin{align*}
\mathbb{F}_{1^{m-1},2;1^{\check m-1},2}
(\mathbf{v};\check{\mathbf{v}})
=
\bigoplus_{n\in\Z}
\F_{1^{m-1},2}
(\mathbf{v}q_3^{-n\delta_{m}})
\boxtimes
\check{\F}_{1^{\check m-1},2}
(\check{\mathbf{v}}\cq_3^{-n\delta_{\check m}})\,.
\end{align*}
In this representation the generators $X_i^\pm(z)$ 
are realized by finite sums of vertex operators. 

The coproduct conjecture predicts in particular the existence of
the tensor product of $\mathbb{F}_{2;2}$ with itself. 
With the aid of screened intertwiners, for odd $M$, we construct directly 
a representation of $\mathcal{A}_{M,2M-2}$ on the space 
\begin{align*}
\mathbb{F}_{2,2;2,2}(v_1,v_2;\check{v}_1,\check{v}_2) 
=\bigoplus_{n_1,n_2\in\Z}
\bigl(\F_2(q_3^{-n_1}v_1)\otimes \F_2(q_3^{-n_2}v_2)\bigr)
\boxtimes
\bigl(\check{\F}_2(\cq_3^{-n_1}\check{v}_1)\otimes 
\check{\F}_2(\cq_3^{-n_2}\check{v}_2)\bigr)\,.
\end{align*}
See Theorem \ref{thm:level2}. 
The generators $X_i^\pm(z)$ are realized by contour integrals,
or after taking residues, by infinite sums of vertex operators
(Jackson integrals).

More generally we expect that the same method yields
representations of $\mathcal{A}_{M,N}$
of the form $\mathbb{F}_{1^a,2^b,3^c;1^{d},2^b,3^c}$ for 
$a,b,c,d\in\Z_{\ge0}$,
where $N=a+d-b-c+(b-c)M$. We sketch the construction in the case $a=c=d=0$
and odd $M$
in Appendix \ref{sec:2222}.
We view the existence of these representations
as a supporting evidence for the coproduct conjecture. 
\medskip

There are a number of questions which remain unanswered. 
\medskip

Little is known about the structure
of the algebra $\mathcal{A}_{M,N}$, such as its graded character or a
 PBW type result. 

In general, the conjectured coproduct of $X^\pm_i(z)$ involves infinite sums 
of generating currents. 
This is similar to the Drinfeld coproduct of 
the  quantum toroidal $\gl_2$ algebra $\E_2$; 
see \eqref{Delta-inf} and discussions around there.
On the other hand, 
the standard \jm{Drinfeld} coproduct $\Delta$ of $\E_2$ 
is given by a finite number of terms in the %Chevalley generators. 
\jm{generating currents $E_i(z),F_i(z), K^\pm_i(z)$ of  $\E_2$.}
Unlike the conjectured coproduct $\Delta_{0,0}$, 
however, $\Delta$ does not preserve the subalgebra 
$\E_1\otimes\check\E_1\subset \E_2$.
This suggests that $\mathcal{A}_{M,N}$ may have another
coproduct of finite type, similar to the one for 
the shifted quantum affine algebras \cite{FT}. 
\footnote{\jm{{\it Note added in proof.}
After finishing the manuscript we have found that
such a finite type corpoduct can be written in the case $M\in\{1,0,-1\}$.
}}

In this paper we treat the simplest situation of gluing 
two $\E_1$ algebras. 
Naturally we expect that this construction carries over to 
gluing $n$ of them, and further to the case 
of general quiver quantum toroidal algebras, see e.g. \cite{NW}.

We hope to return to these issues in the future.
\medskip

The paper is planned as follows.
In Section \ref{sec:gl1} we introduce our convention about the quantum
toroidal $\gl_1$ algebra. We include the vector and the Fock 
representations, and the intertwiners relating a Fock module
and its tensor product with a vector representation. 
We also touch upon the currents obtained by ``fusion'' of the basic currents
$e(z),f(z)$. In Section \ref{sec:AlgAMN} we define the algebra 
$\mathcal{A}_{M,N}$ by generators and relations. We comment on its
elementary properties including automorphisms and gradings. 
We then state our conjecture on the coproduct. 
In Section \ref{sec:rep} we discuss representations of $\mathcal{A}_{M,N}$. 

The text is followed by appendices. 
In Appendix \ref{sec:FermionicR} we give a characterization of 
the $\E_1$ R matrix acting on $\F_1(v_1)\otimes\F_2(v_2)$ using fermions. 
In Appendix \ref{sec:Proof} we prove Theorem \ref{thm:level2}.
In Appendix \ref{sec:2222} we sketch the main steps of constructing 
representations $\mathbb{F}_{2^n;2^n}$.

\vskip1cm

\noindent{\it Notation.}\quad
We use the standard symbols for infinite products
$(z_1,\ldots,z_r;q)_\infty=\prod_{s=1}^r\prod_{j=0}^\infty(1-z_s q^j)$
and  $\theta_q(x)=(x,qx^{-1},q;q)_\infty$.
For $n\in\Z$ we set $(x;q)_n=(x;q)_\infty/(q^nx;q)_\infty=1/(q^nx;q)_{-n}$.

We fix non-zero complex numbers
$q_1,q_2,q_3$ such that $q_1q_2q_3=1$.
We assume that $q_1,q_2$ are generic, meaning
that $q_1^aq_2^b=1$ with $a,b\in\Z$ if and only if $a=b=0$. 
We fix a choice of logarithms $\log q_i$ such that $\log q_1+\log q_2+\log q_3=0$ and define
$q_i^\alpha=e^{\alpha\log q_i}$ for all complex numbers $\alpha$. 

We set for $i=1,2,3$
\begin{align}
s_i&=q_i^{1/2}, 
\quad c_i=-(s_j-s_j^{-1})(s_k-s_k^{-1})\quad 
\text{for $\{i,j,k\}=\{1,2,3\}$}\,,
\label{ci}\\
\beta&=-\frac{\log q_3}{\log q_1}\,.\nn
\end{align}

We use the normal ordering symbol $:~:$ by 
applying the usual rule, bringing creation 
operators to the left and annihilation operators to the right. 
Formally, let 
$\{a_{i,r}\}_{i\in I\,, r\in \Z\setminus\{0\}}$ be a set of 
generators of a Heisenberg algerba satisfying 
$[a_{i,r},a_{j,s}]=B_{i,j}r\delta_{r+s,0}$,
where $B_{i,j}\in\C$. 
For a monomial $A$ in the letters $\{a_{i,r}\}$ 
we define $:A:$ inductively setting
$:1:=1$, $:a_{i,r}A:=a_{i,r}A$ for $r<0$ and $:a_{i,r}A:=Aa_{i,r}$
for $r>0$, and extending it by linearity. 
Some formulas involve also ``zero mode'' operators:
$P,e^Q$ satisfying $Pe^Q=e^Q(P+\gamma)$ with some $\gamma\in\C$.
In this case we bring $e^Q$ to the left and $P$ to the right.

\section{Quantum toroidal $\gl_1$ algebra.}\label{sec:gl1}

In this section we introduce our notation concerning
the quantum toroidal $\gl_1$ algebra and recall some well known facts.  

\subsection{Conventions}

Let
\begin{align*}
& g(z,w)=(z-q_1w)(z-q_2w)(z-q_3w)\,,\\
&\kappa_r=(1-q_1^r)(1-q_2^r)(1-q_3^r), \qquad  r\in \Z\,.
\end{align*} 

The quantum toroidal $\gl_1$ algebra $\E_1=\E_1(q_1,q_2,q_3)$ is 
generated by elements $e_j,f_j, h_r$,
where $j\in\Z$, $r\in\Z\backslash\{0\}$,
and invertible central elements $C$, $\psi_0$. 
In terms of the generating series
\begin{align*}
&e(z) =\sum_{j\in \Z} e_{j}z^{-j}, \quad 
f(z) =\sum_{j\in\Z} f_{j}z^{-j}, \quad 
\psi^{\pm}(z) = \psi_0^{\pm1}
\exp\bigl(\sum_{r=1}^\infty \kappa_r h_{\pm r}z^{\mp r}\bigr)\,,
\end{align*}
the defining relations read as follows:
\begin{align*}
&[h_{r},h_{s}]=
\delta_{r+s,0}\,\frac{1}{r} \frac{C^{r}-C^{-r}}{\kappa_r}\,,
\\ 
&g(C^{(1\pm1)/2}z,w)\psi^\pm(z)e(w)+g(w,C^{(1\pm1)/2}z)e(w)\psi^\pm(z)=0
\,,\\ 
&g(w,C^{(1\mp1)/2}z) \psi^\pm(z)f(w)+g(C^{(1\mp1)/2}z,w)f(w)\psi^\pm(z)=0\,,
\\
&[e(z),f(w)]=\frac{1}{\kappa_1}
\bigl(\delta(C w/z)
\psi^+(w)
-\delta(Cz/w) 
\psi^-(z)\bigr),\\
&g(z,w)e(z)e(w)+g(w,z)e(w)e(z)=0, \\
&g(w,z)f(z)f(w)+g(z,w)f(w)f(z)=0,\\
&\mathop{\on{Sym}}_{z_1,z_2,z_3}\, z_1 z_2^2 [e(z_1),[e(z_2),e(z_3)]]=0,\\
&\mathop{\on{Sym}}_{z_1,z_2,z_3}\, z_1 z_2^2 [f(z_1),[f(z_2),f(z_3)]]=0\,,
\end{align*}
where $\Sym$ stands for symmetrization in the variables indicated.

We use the coproduct 
\begin{align}
&\Delta e(z)=e(z)\otimes 1+\psi^-(z)\otimes e(C_1z)\,,\nn\\ 
&\Delta f(z)=f(C_2z)\otimes \psi^+(z)+1\otimes f(z)\,,\nn\\ 
&\Delta\psi^-(z)=\psi^-(z)\otimes\psi^-(C_1z)\,,\label{coproE1}\\
&\Delta\psi^+(z)=\psi^+(C_2z)\otimes\psi^+(z)\,,\nn\\
&\Delta C=C\otimes C\,,\nn
\end{align}
where $C_1=C\otimes 1$, $C_2=1\otimes C$, the counit
$\epsilon\, e(z) = \epsilon\, f(z) = 0$, $\epsilon \,\psi^\pm (z) = 1$, $\epsilon\, C = 1$, and the antipode
\begin{align}
&Se(z)=-\psi^-(C^{-1}z)^{-1}e(C^{-1}z), \nn\\
&Sf(z)=-f(C^{-1}z) \psi^+(C^{-1}z)^{-1},\label{antipode} \\ 
&S\psi^\pm(z)=\psi^\pm(C^{-1}z)^{-1}, \nn \\
&S\, C=C^{-1}. \nn
\end{align}

Algebra $\E_1$  has a $\Z$-grading, denoted $\deg$, given by the assignment
\begin{align}\label{hom grading}
&\deg e_j=\deg f_j=j\,,\quad \deg h_r=r\,,\quad 
\deg x=0\,\quad \text{for $x=C,\psi_0$.}
\end{align}

We say that an $\E_1$ module 
has level $\ell\in\C^\times$ 
if $C$ acts by the scalar $\ell$.

\subsection{Vector representations}

Set
\begin{align}
&\omega(x)=\frac{1-q_2x}{1-x}\frac{1-q_3x}{1-q_2q_3x}
\label{omega1}\,,
\end{align}
which satisfies
\begin{align*}
\omega(x)=\omega(q_1/x)\,,\quad 
\jm{\frac{\omega(q_1w/z)}{\omega(w/z)}}
=-\frac{g(w,z)}{g(z,w)} \,.
\end{align*}
For a rational function $f(x)$, we let
$f^\pm(x)$ stand for its Laurent series
expansions at $x^{\pm1}=0$.

The vector representation $V_1(v)$ 
of color $1$ and parameter $v\in\C^\times$
is defined on the underlying vector space with basis 
$\{\ket{i,v}\mid i\in\Z\}$.
The generators act as
\begin{align}
&(s_1^{-1}-s_1)e(z)\ket{i,v}=\delta(q_1^iv/z)\ket{i+1,v}\,,\nn\\
&(s_1^{-1}-s_1)f(z)\ket{i,v}=\delta(q_1^{i-1}v/z)\ket{i-1,v}\,,
\label{vector1}\\
&\psi^\pm(z)\ket{i,v}=\omega^\pm(q_1^iv/z)\ket{i,v}\,,\nn
\end{align}
and $C=1$, $\psi_0=1$.

We also view $V_1(v)$ as a right module, denoted $V^R_1(v)$,  
via the anti-automorphism $e(z)\leftrightarrow f(z)$, 
$\psi^\pm(z)\leftrightarrow \psi^\pm(z)$ which exists for $C=1$.

Vector representations $V_c(v)$ and $V^R_c(v)$ of color $c\in\{1,2,3\}$
are defined by interchanging the role of $q_1$ with $q_c$. 
\medskip

\begin{rem}\label{vector1-}
{\rm The same formula \eqref{vector1} on the vector space with 
basis $\{\ket{i,v}\mid i \in\Z+1/2\}$ defines a representation $V^-_1(v)$. 
The map $V^-_1(v)\to V_1(s_1v)$ sending $\ket{i,v}\mapsto \ket{i-1/2,s_1v}$
is an isomorphism of $\E_1$ modules.
\qed} 
\end{rem}

\subsection{Fock representations}
The Fock representation $\F_2(v)$ 
of color $2$ and parameter $v\in\C^\times$
is defined on the irreducible representation 
$\C[h_r\mid r<0]$ of the Heisenberg algebra generated by $\{h_r\}_{r\neq0}$
with $C=s_2$.
The generators $e(z),f(z)$ act by vertex operators, 
\begin{align}
&e(z)\mapsto -c_2^{-1}v\, \xi^{+}(z)\,,
\label{E1Fock-e}\\
&\xi^{+}(z)=
\exp\Bigl(\sum_{r>0}(1-q_1^r)(1-q_3^r)h_{-r}z^r\Bigr)
\exp\Bigl(\sum_{r>0}(1-q_1^{-r})(1-q_3^{-r})s_2^{-r} h_{r}z^{-r}\Bigr)\,,
\nn\\
&f(z)\mapsto c_2^{-1}v^{-1}\, \xi^{-}(z)\,,
\label{E1Fock-f}\\
&\xi^{-}(z)=
\exp\Bigl(-\sum_{r>0}(1-q_1^{r})(1-q_3^{r}) s_2^{r} h_{-r}z^r\Bigr)
\exp\Bigl(-\sum_{r>0}(1-q_1^{-r})(1-q_3^{-r}) h_{r}z^{-r}\Bigr)\,,
\nn
\end{align}
and $C=s_2$, $\psi_0=1$.
Note the identities
\begin{align}
&{\psi^-(z)}^{-1}\xi^+(z)=:{\xi^-(s_2z)}^{-1}:\,,
\quad
\xi^-(z){\psi^+(z)}^{-1}=:{\xi^+(s_2z)}^{-1}:\,.
\label{psi e psi f}
\end{align}

Fock representation $\F_c(v)$ of color $c\in\{1,2,3\}$ is defined by
interchanging the role of $q_2$ with $q_c$.
\medskip

Dealing with Fock spaces with different spectral parameters, 
it will be convenient to use zero mode operators. 
For each $v\in\C^{\times}$ we fix a highest weight vector 
$\ket{\emptyset}_{\F_2(v)}$ of $\F_2(v)$ and
introduce
\begin{align*}
e^Q:\F_2(v)\to\F_2(vq_3^{-1})\,,\quad P:\F_2(v)\to\F_2(v)\,,
\quad 
\end{align*}
which commute with the action of $\psi^{\pm}(z)$ and satisfy
\begin{align*}
e^Q\ket{\emptyset}_{\F_2(v)}=\ket{\emptyset}_{\F_2(v q_3^{-1})}\,,\quad 
P\ket{\emptyset}_{\F_2(v)}=\lambda \ket{\emptyset}_{\F_2(v)}\,
\quad \text{for $v=q_1^\lambda$}.
\end{align*}
We have $[P,e^Q]=\beta e^Q$ (we recall that $q_3=q_1^{-\beta}$).

\medskip

\subsection{Intertwiners.}\label{subsec:Intertwiner}

For each permutation $(a,b,c)$ of $(1,2,3)$
there exist homomorphisms of left $\E_1$ modules of the form
\begin{align*}
&V_{a}(z)\otimes\F_b(v)\longrightarrow \F_b(q_c^{-1}v)\,,\\ 
&\F_b(v)\longrightarrow \F_b(q_cv)\otimes V_a(z)\,.
\end{align*}
We shall refer to them as ``intertwiners''. 
Explicitly they are given by vertex operators, see e.g. 
\cite{AFS}, \cite{Z}, \cite{FOS}
\footnote{The parameters used in these references
are $q=q_1$, $t^{-1}=q_3$, $\gamma=s_2$.}.
For definiteness we take $a=1$, $b=2$ below and present the formulas.

Define
\begin{align}
&\Phi(z)=e^{Q}z^{ P+\beta/2}\Phi^{osc}(z)\,,
\label{Phiosc}\\
&\Phi^{osc}(z)=
\exp\Bigl(-\sum_{r>0}(1-q_3^r)h_{-r}z^r\Bigr) 
\exp\Bigl(-\sum_{r>0}(1-q_3^{-r})s_2^{-r} h_{r}z^{-r}\Bigr)\,,
\nn\\
&\Phi^*(z)=e^{-Q}z^{-P+\beta/2}\Phi^{*osc}(z)\,,
\label{Phiosc*}\\
&\Phi^{*osc}(z)=
\exp\Bigl(\sum_{r>0}(1-q_3^r)s_2^r h_{-r}z^r\Bigr) 
\exp\Bigl(\sum_{r>0}(1-q_3^{-r})h_{r}z^{-r}\Bigr)\,.
\nn
\end{align}
Then the following maps are homomorphisms of left 
$\E_1$ modules for $v\in\C^\times$:
\footnote{Due to the infinite sum,  
\eqref{F2F2V1} is only a formal expression.  
This causes a problem when we compose intertwiners, see Remark
\ref{rem:screened VO}.}
\begin{align}
&V_1(z)\otimes \F_2(v)\longrightarrow \F_2(q_3^{-1}v)\,,
\quad
\ket{i,z}
\otimes w\mapsto (s_3^{-1}v)^i\Phi^{osc}(q_1^iz)w\,, 
\label{V1F2F2}
\\
&\F_2(v)\longrightarrow \F_2(q_3 v)\otimes V_1(z)\,,
\quad w\mapsto \sum_{i\in\Z} 
(s_3v)^{-i} \Phi^{*osc}(q_1^iz) w\otimes\ket{i,z}
\,.
\label{F2F2V1}
\end{align} 
Equivalently, the maps
\begin{align}
&V_1(z) \longrightarrow 
\mathrm{Hom}_{\C}\bigl(\F_2(v),\F_2(vq_3^{-1})
\bigr)
\,,\quad
\ket{i,z}
\mapsto \Phi(q_1^iz)\,,
\label{IntPhi}
\\
&V^R_1(z) 
\longrightarrow 
\mathrm{Hom}_{\C}\bigl(\F_2(v),\F_2(vq_3)\bigr)\,,
\quad
\ket{i,z}
\mapsto \Phi^*(q_1^iz)\,,
\label{IntPhi*}
\end{align}
are homomorphisms of left/right $\E_1$ modules.
Here the Hom spaces are endowed with the left adjoint action $\mathrm{ad}$ 
and the right adjoint action $\mathrm{ad}^R$ of $x\in\E_1$:
\begin{align}\label{ad-act}
&\mathrm{ad}(x)A =\sum_s x_s' A S\bigl(x_s''\bigr)\,,
\quad B\, \mathrm{ad}^R(x)=\sum_s S\bigl(x_s'\bigr)B x_s''\,,
\end{align}
where $A\in \mathrm{Hom}_{\C}\bigl(\F_2(v),\F_2(vq_3^{-1})\bigr)$, 
$B\in \mathrm{Hom}_{\C}\bigl(\F_2(v),\F_2(vq_3)\bigr)$,
 $\Delta x=\sum_s x_s'\otimes x_s''$, and
$S$ denotes the antipode.

\subsection{Fused currents}\label{sec:fused currents}

Quite generally, let $W=\bigoplus_{d\in\Z} W_d$ 
be a graded vector space such that 
\begin{align}\label{admissible}
\dim W_d<\infty\,, \quad W_d=0 \quad \text{for $d\gg0$}.  
\end{align}
A graded current is a formal series
$a(z)=\sum_{k\in\Z} a_k z^{-k}$, $a_k\in\End(W)$, 
satisfying $a_{k}W_d\subset W_{d+k}$ for all $d,k$. 

Suppose $\{a_i(z)\}_{i\in I}$ is a collection of 
graded currents satisfying the relations
\begin{align}
g_{i,j}(z,w)a_i(z)a_j(w)=g_{j,i}(w,z)a_j(w)a_i(z)\,, 
\qquad i,j\in I\,,
\label{formal series relations}
\end{align}
with some homogeneous polynomials $g_{i,j}(z,w)$.
Then each matrix element of 
both sides of \eqref{formal series relations} is 
a Laurent polynomial.
In particular, each matrix element of $a_i(z)a_j(w)$ 
is a rational function with poles at most at $g_{i,j}(z,w)=0$.
Hence if $g_{i,j}(z,tz)\neq 0$ for some $t\in\C^\times$,
then  $a_i(z)a_j(tz)$ is a well defined graded current. 
We say that it is obtained by fusion of $a_i(z)$ and $a_j(tz)$. 
Provided $g_{i,j}(z,t_1z)g_{i,k}(z,t_1t_2z)g_{j,k}(z,t_2z)\neq0$,
we have
\begin{align*}
\bigl(a_i(z)a_j(t_1z)\bigr) a_k(t_1t_2z)
=a_i(z)\bigl(a_j(t_1z)a_k(t_1t_2z)\bigr)\,.
\end{align*}
Under this condition, one can unambiguously define 
fusion of several graded currents. 

\begin{rem}\label{rem:admissible}
{\rm Consider a graded vector space 
$W=\bigoplus_{(d,\alpha)\in\Z\times\Z^m}W_{d,\alpha}$ such that 
\begin{align*}
&\dim W_{d,\alpha}<\infty\,,\\
&\text{for any $\alpha$ there exists a $d(\alpha)$ such that
$W_{d,\alpha}=0$ for $d>d(\alpha)$.}
\end{align*}
A current $a(z)=\sum_{k\in\Z}a_kz^{-k}$ is graded and weighted with 
weight $\lambda$
if $a_{k}W_{d,\alpha}\subset W_{d+k,\alpha+\lambda}$ for all 
$k, d,\alpha$. 
In this setting fusion of graded and weighted currents make sense. 
\qed} 
\end{rem}

We say that a $\deg$-graded $\E_1$ module $W$ is 
admissible if \eqref{admissible} is satisfied.

On admissible $\E_1$ modules, 
the following fused currents are well-defined.
For $r>0$ let
\begin{align*}
&e^{(r)}_{q_1}(z)=c_1^r e(z)e(q_1z)\cdots e(q_1^{r-1}z)\,,\\ 
&f^{(r)}_{q_1}(z)=(-c_1)^rf(q_1^{r-1}z)\cdots f(q_1z)f(z)\,,\\
&\psi_{q_1}^{\pm,(r)}(z)
=\psi^\pm(z)\psi^\pm(q_1 z)\dots \psi^\pm(q_1^{r-1}z)\,,
\end{align*}
where $c_1$ is given in \eqref{ci}. 
We set $e^{(0)}_{q_1}(z)=f^{(0)}_{q_1}(z)=\psi_{q_1}^{\pm,(0)}(z)=1$.
Modifying these currents we introduce currents $k^\pm_r(z)$ which will 
play a role in the next section.

To this end, we adjoin a split central invertible element $K$ to $\E_1$
and consider the algebra
\begin{align*}
\E_1[K]=\E_1(q_1,q_2,q_3)\otimes_\C \C[K,K^{-1}]\,.
\end{align*}
We extend the Hopf structure and $\deg$ by setting
\begin{align*}
\Delta K=K\otimes K\,,\quad \epsilon K=1\,,\quad
S K=K^{-1}\,,\quad
\deg K=0\,. 
\end{align*}
We may view the vector and the Fock representations also as 
representations of $\E_1[K]$ by letting $K$ act by any scalar.
Later on we make a specific choice
in Section \ref{sec:s2s2}, \eqref{Ks2P}.
Accordingly we define
\begin{align*}
&K= \exp\bigl(-\log s_a\log s_b/\log s_1\bigr)
\quad \text{on $V_c(v)$, $V^R_c(v)$
},
\\
&K=\exp\bigl(\log s_c \log v/\log q_1\bigr)
\hspace{28pt}
\text{on $\F_c(v)$}\,,
\end{align*}
where $(a,b,c)$ is a permutation of $(1,2,3)$.
Then the intertwiners \eqref{IntPhi}, \eqref{IntPhi*} are homomorphisms
of $\E_1[K]$ modules. 

Let $k_0^\pm(z)$ be series in $z^{\mp1}$ which are solutions of difference equations
\begin{align}\label{k0}
k_0^\pm(q_1z)=\psi_0^{\mp1}\psi^\pm(z)k_0^\pm(z),  
\qquad k_0^+(\infty)=K
,
\quad  k_0^-(0)=K^{-1}.
\end{align}
Explicitly we have
\begin{align}\label{k0 expl}
k_0^\pm(z)&=K^{\pm1} \exp\Bigl(\sum_{r>0}\frac{\kappa_r}{q_1^{\mp r}-1}
h_{\pm r}z^{\mp r}\Bigr)\,.
\end{align}

For $r\in\Z_{\geq 0}$, define
\begin{align}\label{k}
k^+_r(z)=f^{(r)}_{q_1}(z)k_0^+(z)\,,\quad 
k^-_r(z)=k_0^-(z)e^{(r)}_{q_1}(z)\,.
\end{align}
Here we chose $q_1$ over $q_2$ and $q_3$, breaking the symmetry. 
This is not reflected in our notation and we hope it does not lead to a confusion.

We quote from \cite{FJM2} some properties of $k^\pm_r(z)$.

\begin{lem}\label{k on fock lem}
    In $\mc F_1(u)$ the operators $k^\pm_r(z)$ with $r\geq 2$ act by zero. In addition, we have
    \begin{align}
        k_0^\pm(s_1z)+k_1^\mp(z)=0.
    \end{align}
\qed
\end{lem}

\begin{lem}\label{lem:DeltaK}
The coproduct of the currents $k_r^\pm(z)$ is given by
\begin{align*}
&\Delta k^+_r(z)=\sum_{r_1,r_2\ge0\atop r_1+r_2=r}
k^+_{r_1}(C_2z)\otimes\psi_0^{r_1} k^+_{r_2}(q_1^{r_1}z)\,,\\
&\Delta k^-_r(z)=\sum_{r_1,r_2\ge0\atop r_1+r_2=r}
\psi_0^{-r_2}k^-_{r_1}(q_1^{r_2}z)\otimes k^-_{r_2}(C_1z)\,.
\end{align*}
\qed
\end{lem}
\medskip

\section{Algebra $\mathcal{A}_{M,N}$.}\label{sec:AlgAMN}
In this section we introduce 
a family of algebras $\calA_{M,N}=\calA_{M,N}(q_1,q_2,q_3)$ depending on 
integers $M,N$, and state some simple properties.

\subsection{Definition}\label{sec:defAMN}

Fix $M,N\in\Z$.  
Define $\cq_1,\cq_2,\cq_3$ by 
\begin{align}
&\cq_1=q_1^{-1}\,,\quad \cq_2=q_2q_1^{-M+1}\,\quad \cq_3=q_3q_1^{M+1}\,.
\label{qcq}
\end{align}
Set $\beta=-\log q_3/\log q_1$, 
 $\check\beta=-\log \cq_3/\log \cq_1$, so that 
\begin{align*}
q_3=q_1^{-\beta}\,,\quad \cq_3=\cq_1^{-\cb}\,,\quad \beta+\cb=M+1\,. 
\end{align*}
Let $\E_1={\E}_{1}(q_1,q_2,q_3)$
and $\check{\E}_1=\E_{1}(\cq_1,\cq_2,\cq_3)$ be
the quantum toroidal $\gl_1$ algebras 
with the parameters given above. 
The generators of $\check{\E}_1$ and the fused currents 
will be written as $\ce(z),\cf(z),\check{\psi}^\pm(z)$, $\check{C}$, 
$\check{k}^{\pm}_r(z)$ and so on. We also use 
$\check \kappa_r=(1-\cq_1^r)(1-\cq_2^r)(1-\cq_3^r)$, $\check s_i=\cq_i^{1/2}$, 
and
$\check{c}_i=-(\check{s}_j-\check{s}_j^{-1})
(\check{s}_k-\check{s}_k^{-1})$, $\{i,j,k\}=\{1,2,3\}$. 

Let $\E_1[K]=\E_1\otimes\C[K,K^{-1}]$ and 
$\check\E_1[\check K]=\check\E_1\otimes\C[\check K,\check{K}^{-1}]$
be their extensions introduced in the previous section.
We set 
\begin{align}
\mathcal{K}_{M,N} =
\E_1[K]\otimes\cE_1[\check{K}]/(\psi_0\otimes 1-1,\, 1\otimes\check{\psi}_0-1\,,\,
C\otimes 1-s_1^{N}\otimes\check{C})\,.
\label{C,check C}
\end{align}

First we consider the case of even $N$. 
The case of odd $N$ will be discussed at the end of this subsection.

Algebra $\calA_{M,N}$ is an extension of the algebra $\mathcal{K}_{M,N}$ 
by elements $X^\pm_{i,k}$, $i,k\in\Z$.
We impose three groups of relations (R1)--(R3) given below.
To write the relations we use the generating series
\begin{align*}
&X^\pm_i(z)=\sum_{k\in\Z}X^\pm_{i,k}z^{-k}
\,.
\end{align*}
\medskip

\noindent{(R1)}\quad
The subalgebra $\mathcal{K}_{M,N}$ acts on $\mathcal{A}_{M,N}$ 
by the left/right adjoint action as in \eqref{ad-act}, 
where $x\in\mathcal{K}_{M,N}$ and $A,B\in\mathcal{A}_{M,N}$.
The first group of relations says how the currents $X^\pm_i(z)$ 
transform under these adjoint actions.
Namely we impose the conditions:
\begin{enumerate}
\item The  map
$V_1(z)\to \oplus_{i\in\Z}\C 
X^+_{-i}(q_1^iz)$, 
$\ket{i,z} 
\mapsto  X^+_{-i}(q_1^iz)$, 
is a homomorphism of left $\E_1[K]$ modules,
\item The map 
$V_1^R(z)\to \oplus_{i\in\Z}\C 
X^-_{i}(q_1^iz)$, 
$\ket{i,z} 
\mapsto  X^-_{i}(q_1^iz)$, 
is a homomorphism of right $\E_1[K]$ modules,
\item The map 
$\check V_1(z)\to \oplus_{i\in\Z}\C  
X^+_{-i}(z)$,
$\ket{i,z}^\vee 
\mapsto  X^+_{-i}(z)$,
is a homomorphism of left $\check \E_1[\check K]$ modules,
\item The map 
$\jm{\check{V}_1^R(z)}\to \oplus_{i\in\Z}\C  
X^-_{i}(z)$, 
$\ket{i,z}^\vee 
\mapsto  X^-_{i}(z)$, 
is a homomorphism of right $\check \E_1[\check K]$ modules.
\end{enumerate}
Here we formally treat $z$ as a complex number.

We use below the function $\omega(x)$ in \eqref{omega1} and 
\begin{align*}
\check{\omega}(x)=\frac{(1-\cq_2x)(1-\cq_3x)}{(1-x)(1-\cq_2\cq_3x)}\,.
\end{align*}
In reality (i)--(iv) mean the following explicit relations.
\begin{align*}
(KX_i^\pm):\quad
&KX^\pm_i(z)K^{-1}=q_2^{\pm\beta/2}X^\pm_i(z)\,,\\
(\check KX_i^\pm):\quad
&\check{K}X^\pm_i(z)\check{K}^{-1}
=\cq_2^{\pm\check\beta/2}X^\pm_i(z)\,,
\\
&\\
(\psi^\pm X_i^+):\quad
&
\psi^\pm(u)X^+_i(z){\psi^\pm(u)}^{-1}=
\omega^\pm\Bigl(\frac{z}{C^{(1\pm1)/2} u}\Bigr)X^+_i(z)\,,
\\
(eX^+_i):\quad 
&e(u)X^+_i(z)-\omega^-\Bigl(\frac{z}{u}\Bigr)X^+_i(z)e(u)=
\frac{1}{s_1^{-1}-s_1}\delta\Bigl(\frac{z}{u}\Bigr)X^+_{i-1}(q_1z)\,,
\\ 
(fX^+_i):\quad 
&f(u)X^+_i(z)-X^+_i(z)f(u)=\frac{1}{s_1^{-1}-s_1}
\delta\Bigl(\frac{q_1^{-1}z}{C u}\Bigr)X^+_{i+1}(q^{-1}_1z)
\psi^+(u)
\,,
\end{align*}
\begin{align*}
(\check\psi^\pm X_i^+):\quad
&\check{\psi}^\pm(u)X^+_i(z){\check{\psi}^\pm(u)}^{-1}
=\check{\omega}^\pm\Bigl(\frac{q_1^{i} z}{\check{C}^{(1\pm1)/2}u}\Bigr)
X^+_i(z)\,,
\\
(\check e X^+_i):\quad 
&\check{e}(u)X^+_i(z)-\check{\omega}^-\Bigl(\frac{q_1^{i}z}{u}\Bigr)
X^+_i(z)\check{e}(u)=
\frac{1}{\check{s}_1^{-1}-\check{s}_1}
\delta\Bigl(\frac{q_1^{i}z}{u}\Bigr)X^+_{i-1}(z)\,,
\\
(\check f X^+_i):\quad 
&\check{f}(u)X^+_i(z)-X^+_i(z)\check{f}(u)
=\frac{1}{\check{s}_1^{-1}-\check{s}_1}
\delta\Bigl(\frac{q_1^{i+1}z}{\check{C}u}\Bigr)
X^+_{i+1}(z)\check{\psi}^+(u)\,,
\end{align*}
\begin{align*}
(\psi^\pm X_i^-):\quad
&{\psi^\pm(u)}^{-1}X^-_i(z)\psi^\pm(u)
=\omega^\pm\Bigl(\frac{z}{C^{(1\mp1)/2}u}\Bigr)X^-_i(z)\,,
\\
(eX^-_i):\quad 
&X^-_i(z)e(u)-e(u)X^-_i(z)=
\frac{1}{s_1^{-1}-s_1}\delta\Bigl(\frac{z}{C q_1u}\Bigr) 
\psi^-(u) X^-_{i-1}(q_1^{-1}z)\,,
\\ 
(fX^-_i):\quad 
&X^-_i(z)f(u)-\omega^+\Bigl(\frac{z}{u}\Bigr)f(u)X^-_i(z)
=\frac{1}{s_1^{-1}-s_1}\delta\Bigl(\frac{z}{u}\Bigr)X^-_{i+1}(q_1z)\,,
\end{align*}
\begin{align*}
(\check\psi^\pm X_i^-):\quad
&{\check{\psi}^\pm(u)}^{-1}X^-_i(z)\check{\psi}^\pm(u)
=\check{\omega}^\pm\Bigl(\frac{q_1^{-i}z}{\check{C}^{(1\mp1)/2}u}\Bigr)X^-_i(z)\,,
\\
(\check eX^-_i):\quad 
&X^-_i(z)\check{e}(u)-\check{e}(u)X^-_i(z)
=\frac{1}{\check{s}_1^{-1}-\check{s}_1}
\delta\Bigl(\frac{q_1^{1-i}z}{\check{C}u}\Bigr)
\check{\psi}^{-}(u)X^-_{i-1}(z)\,,
\\
(\check f X^-_i):\quad 
&X^-_i(z)\check{f}(u)-\check{\omega}^+\Bigl(\frac{q_1^{-i}z}{u}\Bigr)
\check{f}(u)X^-_i(z)=
\frac{1}{\check{s}_1^{-1}-\check{s}_1}
\delta\Bigl(\frac{q_1^{-i}z}{u}\Bigr)X^-_{i+1}(z)
\,.
\end{align*}
\medskip

\noindent{(R2)} The second group is a set of quadratic relations between 
$X^\pm_i(z)$ and $X^\pm_j(w)$,
\begin{align}
&\gamma_{i,j}(z,w)X^+_i(z)X^+_j(w) =(-1)^{i-j-1}\gamma_{j,i}(w,z)X^+_j(w) X^+_i(z)\,,
\label{X+X+}
\\
&\gamma_{i,j}(w,z)X^-_i(z)X^-_j(w) =(-1)^{i-j-1}\gamma_{j,i}(z,w)X^-_j(w) X^-_i(z)\,.
\label{X-X-}
\end{align}
For $M\ge0$, the coefficients are given explicitly as follows.

If $i=j$, then
\begin{align}
&\gamma_{i,i}(z,w)=\prod_{r=0}^{M-1}(z-q_2q_1^{-r}w)\,. \label{gii}
\end{align}
If $0<j-i\le M$, then
\begin{align}
&\gamma_{i,j}(z,w)=\prod_{r=1}^{j-i-1}(z-q_1^rw)
\prod_{r=0}^{-j+i+M-1}(z-q_2q_1^{-r}w)\,,
\label{gij1}\\
&\gamma_{j,i}(w,z)=(-1)^{j-i+1}
\frac{\prod_{r=0}^{j-i+M-1}(w-q_2q_1^{-r}z)}{\prod_{r=0}^{j-i}(w-q_1^{-r}z)} 
\times \check{q}_2^{-j+i}
\nn\,.
\end{align}
If $j-i>M$, then
\begin{align}
&\gamma_{i,j}(z,w)=\prod_{r=1}^{j-i-1}(z-q_1^rw)\,,
\label{gij2}\\ 
&\gamma_{j,i}(w,z)=
\frac{\prod_{r=0}^{j-i+M-1}(w-q_2q_1^{-r}z)
\prod_{r=0}^{j-i-M-1}(w-q_3q_1^{-r}z)}
{\prod_{r=0}^{j-i}(w-q_1^{-r}z)}
\times
(-1)^{M-1}(q_2\check{q_2})^{-\frac{M}{2}}q_1^{(j-i)(j-i-1)/2}\,.
\nn
\end{align}
For $M\le 0$, we replace $M\to -M$,
$q_2\leftrightarrow q_3$ and $\check{q}_2\leftrightarrow \check{q}_3$.

We say that a rational function $f(x)$ is balanced if $f(0),f(\infty)$ exist
and $f(0)f(\infty)=1$. 
In all cases $\gamma_{j,i}(w,z)/\gamma_{i,j}(z,w)$ is a balanced 
function in $w/z$. 

The coefficients in \eqref{X+X+}, \eqref{X-X-}
are to be understood as expansions in $|w/z|\ll 1$ in the left hand side, 
and in $|z/w|\ll 1$ in the right hand side.

\medskip

\noindent{(R3)} The last group of relations is between $X^+_i(z)$ and $X^-_j(w)$.

In the notation $[X,Y]=XY+(-1)^{M}YX$ we have
\begin{align}
[X^+_i(z),X^-_j(w)]&=
\sum_{r+\check{r}=i+j+N/2\atop r,\check{r}\in\Z_{\ge0}}
\delta\Bigl(\frac{C w}{q_1^rz}\Bigr) 
\, k_r^+(q_1^{-r}w)
\check{k}_{\check{r}}^+(\check{q}_1^{-\check{r}}q_1^{-j}w) 
\label{X+X-}\\
&-(-1)^{N} 
\sum_{r+\check{r}=-i-j+N/2\atop r,\check{r}\in\Z_{\ge0}}
\delta\Bigl(\frac{C z}{q_1^rw}\Bigr)
\,k^-_r(q_1^{-r}z) 
\check{k}_{\check{r}}^-(\check{q}_1^{-\check{r}}q_1^iz)   
\,.\nn
\end{align}
\medskip

For odd $N$, we define $\mathcal{A}_{M,N}$ to be
an extension of the algebra $\mathcal{K}_{M,N}$ by 
\begin{align*}
X^+_i(z)\,,\quad i\in \Z+\frac{1}{2}\,,\qquad X^-_i(z)\,,\quad i\in\Z\,. 
\end{align*}
Except for this change of labels, all relations ($KX^\pm_i$)--($\check f X^-_i$)
in (R1) and (R2), (R3) are kept as written.
(This means that in the condition (i) of (R1) we change $V_1(z)$ 
to $V_1^-(z)=\oplus_{i\in\Z+1/2}\C\ket{i,z}$, 
and similarly for (iii), see Remark \ref{vector1-}.)

\subsection{Remarks about $\mathcal{A}_{M,N}$ }\label{sec:remark AMN}

We give here several remarks concerning $\mathcal{A}_{M,N}$.

Algebra $\mathcal{A}_{M,N}$ has a $\Z\times \Z$-grading which we call 
degree (denote by $\deg$) and weight (denoted $\wt$).  
For an element $x\in\mathcal{K}_{M,N}$, $\deg x$ is given by 
\eqref{hom grading} and its analog for  $\check \E_1[\check K]$. We set also
\begin{align*}
\wt x=0\,,\qquad x\in \mathcal{K}_{M,N}\,.
\end{align*}
In addition we set
$$
\deg X^\pm_{i,k}=k\,,\quad \wt X^\pm_{i,k}=\pm1\,.
$$
We say a representation of $\mathcal{A}_{M,N}$ is admissible
if it is admissible in the sense of Remark \ref{rem:admissible}
with the above degree and weight.

As mentioned already,  $\mathcal{A}_{M,N}$
is not an algebra in the usual sense because 
$k^\pm_r(z)$, $\check k^\pm_r(z)$ with $r\neq 0$
in the right hand side of \eqref{X+X-} involve 
infinite sums of generators $e_i,f_i$. 
They make sense only on admissible representations.

When $M$ is even, algebra $\mathcal{A}_{M,N}$ should be treated as a superalgebra. 
The generators $X^\pm_i(z)$ are odd and those in $\mathcal{K}_{M,N}$ are even. 
The bracket $[X^+_i(z),X^-_j(w)]$ in (R3) is a commutator in the sense of
superalgebras.
Tensor products of operators and modules should also be taken accordingly.

\medskip

We call $N$ {\it shift parameter}. 
Apart from the relation $C=s_1^N \check{C}$
between the central elements, it enters only the relation (R3).
In particular, if $N<0$, then we have
\begin{align*}
[X^+_i(z),X^-_j(w)]=0\quad \text{for $\frac{N}{2}<i+j<-\frac{N}{2}$}.
\end{align*}
Such commutativity is one of the features of
(anti-dominantly) shifted quantum affine algebras \cite{FT,FPT}.

The labeling convention for $X^\pm_i(z)$ for odd $N$ is not quite symmetric. 
This leads us to consider another version $\mathcal{A}^-_{M,N}$, defined as an
extension of $\mathcal{K}_{M,N}$ by elements
\begin{align*}
& X^+_i(z)\,,\quad i\in \Z+\frac{1}{2}\,,\qquad X^-_i(z)\,,\quad i\in\Z+\frac{1}{2}
\qquad \text{for even $N$},\\
& X^+_i(z)\,,\quad i\in \Z\,,\hspace{46pt} X^-_i(z)\,,\quad i\in\Z+\frac{1}{2}\,
\qquad \text{for odd $N$},
\end{align*}
and keeping all relations unchanged. 
It turns out that we do not get a new algebra.
\begin{lem}
There is an algebra isomorphism $\mc A_{M,N}^{-}\to \mc A_{M,N}$ given by
\begin{align}
X^+_{i}(z)&\mapsto X^+_{i+ 1/2}(z), \qquad  X^-_{i}(z)\mapsto X^-_{i-1/2}(z), 
\label{shift lem}\\
e(z)&\mapsto e(z), \hspace{63pt} f(z)\mapsto f(z), 
\nn\\
\check{e}(z)&\mapsto \check{e}(s_1z), \hspace{54pt} 
\check{f}(z)\mapsto \check{f}(s_1z).\nn
\end{align}
\qed
\end{lem}

Due to the relation (R1), algebra $\mathcal{A}_{M,N}$ is 
generated by $\mathcal{K}_{M,N}$ and $X^\pm_0(z)$ for even $N$ or
$X^+_{1/2}(z)$, $X^-_0(z)$ for odd $N$. 
Explicitly we have the following recursion relations.
\begin{align}
X^+_{i-1}(z)
&=c_1 X^+_{i}(q_1^{-1}z)e(q_1^{-1}z)
\label{X+e}\\ 
&=\check c_1 X^+_{i}(z)\check e(q_1^{i}z)
\nn\\ 
X^+_{i+1}(z)
&=-c_1f(C^{-1}z){\psi^+(C^{-1}z)}^{-1}X_i^+(q_1z)
\label{X+f}\\
&=-\check c_1\check f(\check C^{-1}q_1^{i+1}z)
{\check \psi^+(\check C^{-1}q_1^{i+1}z)}^{-1}X_i^+(z)
\nn\\
X^-_{i-1}(z)
&=c_1 X^-_i(q_1z){\psi^-(C^{-1}z)}^{-1}e(C^{-1}z)
\label{X-e}\\ 
&=\check c_1 X^-_i(z){\check\psi^-(\check C^{-1}q_1^{-i+1}z)}^{-1}
\check e(\check C^{-1}q_1^{-i+1}z)
\nn\\
X^-_{i+1}(z)
&=-c_1 f(q_1^{-1}z)X^-_i(q_1^{-1}z)
\label{X-f}\\
&=-\check c_1 \check f(q_1^{-i}z)X^-_i(z)\,.
\nn
\end{align}
These formulas are obtained by comparing residues in relations
 ($eX^+_i$)--($\check{f}X^-_i$).

Algebra $\mathcal{A}_{M,N}$ has several natural automorphisms.

For each $a\in\C^\times$ there is a shift automorphism 
sending \jm{$X(z) \to X(az)$
%, where  $X(z)$ is any of the currents 
for all $X(z)$ from 
$X^\pm_i(z),e(z), f(z), \psi^{\pm}(z), \check e(z), \check f(z), \check \psi^{\pm}(z)$.}

For each $c\in\C^\times$ there is an automorphism sending
$$
X^+_i(z)\mapsto cX^+_i(z)\,,\qquad X^-_i(z)\mapsto c^{-1}X^-_i(z)\,,
\qquad x\mapsto x \quad \text{for $x\in \mathcal{K}_{M,N}$.}
$$

There is an automorphism mapping
\begin{align}\label{relabel E and F}
&X^+_i(z)\mapsto X^+_{i+1}(z)\,,\qquad X^-_i(z)\mapsto X^-_{i-1}(z)\,,
\\
&x(z)\mapsto x(z) \quad \text{for $x=e,\ f,\ \psi^{\pm}$}\,,
\qquad \check x(z)\mapsto \check x(q_1z) \quad 
\text{for $\check x=\check e,\ \check f,\
\check\psi^\pm$}\,.
\nn
\end{align}

We have an algebra isomorphism from 
$\mathcal{A}_{M,N}(q_1,q_2,q_3)$ to 
$\mathcal{A}_{M,N}(\check q_1,\check q_2,\check q_3)$ sending
$$
X^+_i(z)\mapsto X^+_i(q_1^iz), \qquad X^-_i(z)\mapsto X^-_i(q_1^{-i}z), 
$$
and interchanging $\E_1[K]$ with 
$\check{\E}_1[\check K]$, namely mapping 
$x(z) \leftrightarrow \check x(z)$, where $x(z)$ is $e(z),f(z)$ or $\psi^\pm(z)$, 
and $C\leftrightarrow \check C$, $K\leftrightarrow \check K$.

Another algebra isomorphism is 
\begin{align}
\mathcal{A}_{M,N}(q_1,q_2,q_3)\longrightarrow \mathcal{A}_{-M,N}(q_1,q_3,q_2)  
\label{Mto-M}
\end{align}
sending each element to an element by the same name.

\subsection{Special cases}\label{sec:N=0}

Two special cases $(M,N)=(1,0)$ and $(M,N)=(0,0)$ are worth mentioning: 
algebra $\mathcal{A}_{1,0}$ is related to the quantum toroidal $\gl_2$ algebra
$\E_2$, 
while 
algebra $\mathcal{A}_{0,0}$ is related to the quantum toroidal $\gl_{1|1}$ algebra
$\E_{1|1}$. 
\medskip

In \cite{FJM2} we studied an algebra $\mathcal{A}_0$, which is 
a slight generalization of 
$\mathcal{A}_{1,0}$ wherein $\psi_0=\check{\psi}_0$ is imposed but is 
not necessarily $1$. 
On admissible modules, $\mathcal{A}_0$ essentially coincides with 
the image of an extension $\E_2[K]$ of $\E_2$ by a grouplike invertible 
(but not central) element $K$;
see Theorem 4.4 and Theorem 4.5 in \cite{FJM2}.
\medskip

With a slight change of notation in \cite{NW},  
we define the quantum toroidal $\gl_{1|1}$ algebra
 $\E_{1|1}=\E_{1|1}(\bq_1,\bq_2)$
as follows. 

Let
\begin{align*}
&g_{0,0}(z,w)=g_{1,1}(z,w)=1\,,\\
&g_{0,1}(z,w)=(z-\bq_1 w)(z-\bq_1^{-1}w)\,,\quad 
g_{1,0}(z,w)=(w-\bq_2 z)(w-\bq_2^{-1} z)\,.
\end{align*}
Algebra $\E_{1|1}$ is a unital associative algebra 
generated by elements $E_{i,k},F_{i,k},H_{i,r}$, invertible elements $\Psi_{i,0}$,
where $i\in\Z/2\Z$, $k\in\Z$, $r\in\Z\backslash\{0\}$, 
and an invertible central element $C$.
In terms of the generating series
\begin{align*}
&E_i(z) =\sum_{k\in \Z} E_{i,k}z^{-k}, \quad 
F_i(z) =\sum_{k\in\Z} F_{i,k}z^{-k}, \quad 
\Psi_i^{\pm}(z) = \Psi_{i,0}^{\pm1}
\exp\bigl(\pm(\bar s_2-\bar s_2^{-1})
\sum_{r=1}^\infty  H_{i,\pm r}z^{\mp r}\bigr)\,,
\end{align*}
the defining relations are
\begin{align*}
&\Psi^\pm_i(z)\Psi^\pm_j(w)=\Psi^\pm_j(w)\Psi^\pm_i(z)\,,
\\
&\frac{g_{i,j}(z,C^{-1}w)}{g_{i,j}(z,Cw)}
\Psi^+_i(z)\Psi^-_j(w)
=\frac{g_{j,i}(C^{-1}w,z)}{g_{j,i}(Cw,z)}
\Psi^-_j(w)\Psi^+_i(z)
\,,
\\
&g_{i,j}(C^{(1\pm1)/2}z,w)\Psi^{\pm}_i(z)E_j(w)
=g_{j,i}(w,C^{(1\pm1)/2}z) E_j(w)\Psi^{\pm}_i(z)
\,,
\\
&g_{j,i}(w,C^{(1\mp1)/2}z)\Psi^{\pm}_i(z)F_j(w)
=g_{i,j}(C^{(1\mp1)/2}z,w)
F_j(w)\Psi^\pm_i(z)
\,,
\\
&g_{i,j}(z,w)E_i(z)E_j(w)+g_{j,i}(w,z)E_j(w)E_i(z)=0\,,
\\
&g_{j,i}(w,z)F_i(z)F_j(w)+g_{i,j}(z,w)F_j(w)F_i(z)=0\,,
\\
&[E_i(z),F_j(w)]=-\delta_{i,j}
\Bigl(\delta\Bigl(\frac{Cw}{z}\Bigr)\Psi^+_i(w)
-
\delta\Bigl(\frac{Cz}{w}\Bigr)\Psi^-_i(z)
\Bigr)\,.
\end{align*}
The Serre relations are yet unknown.

Let $\E_{1|1}[K]=\E_{1|1}\rtimes\C[K,K^{-1}]$ be the algebra obtained 
by adjoining an invertible element $K$ 
satisfying 
\begin{align*}
&KE_0(z)K^{-1}=q_2^{-\beta/2}E_0(z)\,,\quad 
KF_0(z)K^{-1}=q_2^{\beta/2}F_0(z)\,,\\
&KE_1(z)K^{-1}=q_2^{\beta/2}E_1(z)\,,\quad 
KF_1(z)K^{-1}=q_2^{-\beta/2}F_1(z)\,,\\
&K\Psi^\pm_i(z)=\Psi^\pm_i(z)K\,.
\end{align*}

Let $\hat{\mathcal{A}}_{0,0}$ be the algebra defined similarly as
$\mathcal{A}_{0,0}$, where \eqref{C,check C} is replaced by
\begin{align*}
\hat{\mathcal{K}}_{0,0} =
\E_1[K]\otimes\cE_1[\check{K}]/(\psi_0\otimes 1-1\otimes\check{\psi}_0\,,\,
C\otimes 1-\jm{1}\otimes\check{C})\,,
\end{align*}
and \eqref{X+X-} is replaced by
\begin{align*}
[X^+_i(z),X^-_j(w)]&=
\psi_0^{-i}
\sum_{r+\check{r}=i+j\atop r,\check{r}\in\Z_{\ge0}}
\delta\Bigl(\frac{C w}{q_1^rz}\Bigr) 
\, k_r^+(q_1^{-r}w)
\check{k}_{\check{r}}^+(\check{q}_1^{-\check{r}}q_1^{-j}w) 
\\
&- \psi_0^{-j}
\sum_{r+\check{r}=-i-j\atop r,\check{r}\in\Z_{\ge0}}
\delta\Bigl(\frac{C z}{q_1^rw}\Bigr)
\,k^-_r(q_1^{-r}z) 
\check{k}_{\check{r}}^-(\check{q}_1^{-\check{r}}q_1^iz)   
\,.
\end{align*}

\begin{prop}
Suppose that 
the parameters of the algebras $\E_{1|1}[K]$ and $\hat{\mathcal{A}}_{0,0}$ 
are related by
\begin{align*}
\bq_1=q_1^{1/2}\,,\quad \bq_2=q_1^{1/2}q_2\,.
\end{align*}
Then there exists a homomorphism 
of algebras $\E_{1|1}[K]\longrightarrow \hat{\mathcal{A}}_{0,0}$
given by
\begin{align*}
&\Psi^\pm_1(z)\mapsto k^\pm_0(z)\check k^\pm_0(z)
\,,
\quad 
\Psi^\pm_0(\bq_1 z)\mapsto \psi_0^{\pm1}
\bigl(k^\pm_0(z)\check k^\pm_0(q_1 z)
\bigr)^{-1}\,,
\\
&E_1(z)\mapsto X^+_0(z)\,,\quad 
E_0(\bq_1 z)\mapsto
\psi_0^{-1} \bigl(k_0^-(z)\check k^-_0(q_1z)\bigr)^{-1}
X^-_{-1}(Cz)\,,
\\
&
F_1(z)\mapsto
X^-_0(z)\,,
\quad
F_0(\bq_1 z) \mapsto X^+_1(Cz)\psi_0 
\bigl(k_0^+(z)\check k^+_0(q_1z)\bigr)^{-1}
\,,
\end{align*}
and $C\mapsto C$, $K\mapsto K$.
\end{prop}
\begin{proof}
The checking is similar to the proof of Theorem 4.4 in \cite{FJM2}.
\end{proof}

Similarly to the case of $\gl_2$, 
we expect the map to be an isomorphism after passing to
the quotient of $\E_{1|1}[K]$ by appropriate Serre relations.

\subsection{Coproduct issues.}\label{sec:corpo}

Quite generally, given an algebra $A$ one can consider its extension by 
a left $A$ module $W$. Namely we take the free algebra $T(W)$ over $W$
and construct the semi-direct product 
algebra $T(W)\rtimes A$.
In addition, if $A$ is a quasi-triangular Hopf algebra,
then one can make it a bialgebra: 

\begin{lem}\label{lem:HofExt}
Let $(A,R)$ be a quasi-triangular Hopf algebra,
$W$ a left $A$ module, and $T(W)$ the tensor algebra over $W$.
Then the semi-direct product algebra
$T(W)\rtimes A$ 
is equipped with 
a bialgebra structure by extending the coproduct and counit on $A$ by
\begin{align*}
\Delta(x)=x\otimes 1+R^{-1} (1\otimes x) R\,, 
\qquad \epsilon(x)=0\,, \quad \text{for $x\in W$}\,.
\end{align*}

Similarly, if $W$ is a right $A$ module, then $A\ltimes T(W)$ 
has a bialgebra structure such that 
\begin{align*}
\Delta(x)=R^{-1} (x\otimes 1) R+1\otimes x\,, 
\qquad \epsilon(x)=0\,, \quad \text{for $x\in W$}\,.
\end{align*}
\end{lem}
\begin{proof}
    The check is straightforward.
\end{proof}

\medskip

Motivated by this observation, we propose the following.

\begin{conj}\label{conj:copro}
For any $M,N_1,N_2\in\Z$, there exists a homomorphism of algebras
\begin{align*}
\Delta_{N_1,N_2}\,&:\,
\calA_{M,N_1+N_2}
\longrightarrow
\calA_{M,N_1}\hat \otimes \calA_{M,N_2}
\,
\end{align*}
with the properties (i)--(iii) below. 
Here  $\hat \otimes$ denotes the completion of the tensor product with respect to the homogeneous grading.

\begin{enumerate}
 \item It satisfies the coassociativity, see Figure 1:
\begin{align*}
(\Delta_{N_1,N_2}\otimes \id)\circ\Delta_{N_1+N_2,N_3}
=(\id\otimes\Delta_{N_2,N_3})\circ\Delta_{N_1,N_2+N_3}\,.
\end{align*}
\item Its restriction to the subalgebra 
$\mathcal{K}_{M,N_1+N_2}$
gives
the standard coproduct \eqref{coproE1}.
\item We have
\begin{align*}
&\Delta_{N_1,N_2} X^+_i(z) 
=X^+_i(z)\otimes 1+ \mathcal{R}^{-1}(1\otimes X^+_i(z))\mathcal{R}\,,\\
&\Delta_{N_1,N_2} X^-_i(z) 
=\mathcal{R}^{-1}(X^-_i(z)\otimes 1)\mathcal{R}+ 1\otimes X^-_i(z)\,,
\end{align*}
where $\mathcal{R}=R\check R$ 
is the product of the universal $R$ matrices
of $\E_1[K]$ and of $\check{\E}_1[\check K]$.
(The notation $\check R$ is not to be confused with 
``permutation times $R$''.)
\end{enumerate}
\end{conj}

\begin{figure}[ht]\label{fig:coas}
\centering
\begin{tikzpicture}
\coordinate (A) at (-3,1.5);
\coordinate (B) at (3,1.5);
\coordinate (C) at (-3,-1.5);
\coordinate (D) at (3.4,-1.5);

\node at (A) {$\calA_{M,N_1+N_2+N_3}$};
\node at (B) {$\calA_{M,N_1+N_2}\hat\otimes\calA_{M,N_3}$};
\node at (C) {$\calA_{M,N_1}\hat\otimes\calA_{M,N_2+N_3}$};
\node at (D) {$\calA_{M,N_1}\hat\otimes\calA_{M,N_2}\hat\otimes \calA_{M,N_3}$};

\node at (0,2) {$\Delta_{N_1+N_2,N_3}$};
\node at (0,-1) {$\id\otimes \Delta_{N_2,N_3}$};
\node at (-4,0){$\Delta_{N_1,N_2+N_3}$};
\node at (4,0) {$\Delta_{N_1,N_2}\otimes\id$};

\draw[->] (-0.7,1.5)--(0.5,1.5);
\draw[->] (-0.5,-1.5)--(0.5,-1.5);
\draw[->] (-2.5,1.)--(-2.5,-1.);
\draw[->] (2.5,1.)--(2.5,-1.);
\end{tikzpicture}
\caption{Coassociativity. }
\end{figure}

At this writing it is not clear to us whether the formulas in (iii) 
are well-defined and compatible with the relations (R1)--(R3).
We shall come back to this point when we discuss
examples of representations of $\mathcal{A}_{M,N}$.

\section{Representations.}\label{sec:rep}

In this section we give examples of representations of $\mathcal{A}_{M,N}$. 
In all cases considered here, the representations are such that
when restricted to $\mathcal{K}_{M,N}$
they decompose to a 
direct sum of tensor products of Fock representations
of the form
$\bigl(\otimes_{k=1}^m\F_{c_k}(v_{i_k})\bigr)\boxtimes
\bigl(\otimes_{l=1}^n\check{\F}_{\check{c}_l}(\check{v}_{j_l})\bigr)$,
where the indices $c_k,\check{c}_l\in\{1,2,3\}$ are fixed and 
 $v_{i_k},\check{v}_{j_l}$ vary. 
The shift parameter $N$ is given by 
$s_1^{N}=\prod_{k=1}^ms_{c_k}/\prod_{l=1}^n\check{s}_{\check{c}_l}$.

From now on, we assume that $|q_1|<1$. 
We say an $\mathcal{A}_{M,N}$ module $W$ 
has level $(\ell,\check \ell)\in(\C^\times)^2$ if 
$C$ and $\check C$ act in $W$ by scalars $\ell$ and $\check \ell$, 
respectively. 
We always have $\check \ell=\ell s_1^{-N}$.
\medskip 

\subsection{Representation $\mathbb{F}_{2;2}$.}\label{sec:s2s2}

We begin by considering a family of representations of 
$\mathcal{A}_{M,M-1}$ of level $(s_2,\check s_2)$.
\jm{In this case, }
$X^\pm_i(z)$ are given by single vertex operators.

We fix $v,\check{v}\in\C^{\times}$ satisfying
\begin{align}
v=q_1^\lambda\,,\quad \check{v}=\cq_1^{\check{\lambda}}\,, 
\quad \lambda+\check{\lambda} \in \Z+\frac{M+1}{2}
\,, 
\label{integral-cond}
\end{align}
and consider the representation of $\mathcal{K}_{M,N}$
\begin{align*}
\mathbb{F}_{2;2}(v;\check{v})=\bigoplus_{n\in\Z}\F_2(q_3^{-n}v)\boxtimes
\check{\F}_2(\cq_3^{-n}\check{v})\,. 
\end{align*}

The action of $K,\check{K}$ can be written as
\begin{align}
&K=s_2^{P}\,,\quad \check{K}=\check{s}_2^{\check{P}}\,.
\label{Ks2P}
\end{align}

We shall also use the $\check{\E}_1[\check K]$  version of 
intertwiners 
\begin{align*}
&\check{\Phi}(z) =e^{\check{Q}} z^{\check{P}+\check{\beta}/2} 
\check{\Phi}^{osc}(z)\,,\quad
\check{\Phi}^*(z) =e^{-\check{Q}} z^{-\check{P}+\check{\beta}/2} 
\check{\Phi}^{*osc}(z)\,,
\end{align*}
obtained by changing $q_i$ to $\cq_i$.

Quite generally, if $X(z),Y(w)$ are vertex operators, 
their product has the form 
\begin{align*}
X(z)Y(w)=\varphi(w/z):X(z)Y(w):
\end{align*}
with some scalar function $\varphi(w/z)$.
We say  $\varphi(w/z)$ is the contraction of $X(z),Y(w)$
and denote it by $\cont{X(z)}{Y(w)}$. 

We give below the list of contractions of intertwiners. 
Note that, 
since $|\cq_1|=|q_1^{-1}|>1$, 
the contractions for $\check{\Phi}(z),\check{\Phi}^*(z)$
are not obtained by changing $q_i$ to $\check{q}_i$; 
we should rather keep $q_1$ and change $\beta\to \check{\beta}$,
$\Phi(z)\to\check{\Phi}^*(z)$,
and
$\Phi^*(z)\to\check{\Phi}(z)$. 

\begin{lem}
The following formulas hold:
\begin{align}
&\cont{\Phi(z)}{\Phi(w)}=z^\beta
\frac{(q_1^{1-\beta} w/z;q_1)_\infty}{(q_1 w/z;q_1)_\infty}\,,
\quad
\cont{\Phi^*(z)}{\Phi^*(w)}=z^\beta
\frac{(w/z;q_1)_\infty}{(q_1^{\beta} w/z;q_1)_\infty}\,,
\label{cont1}\\
&\cont{\Phi^*(z)}{\Phi(w)}=\cont{\Phi(z)}{\Phi^*(w)}
=z^{-\beta}
\frac{(q_1^{(1+\beta)/2} w/z;q_1)_\infty}
{(q_1^{(1-\beta)/2} w/z;q_1)_\infty}\,,
\label{cont2}\\
&\cont{\check{\Phi}(z)}{\check{\Phi}(w)}
=z^{\check{\beta}}\frac{(w/z;q_1)_\infty}
{(q_1^{\check{\beta}}w/z;q_1)_\infty}\,,
\quad
\cont{\check{\Phi}^*(z)}{\check{\Phi}^*(w)}
=z^{\check{\beta}}\frac{(q_1^{1-\check{\beta}}w/z;q_1)_\infty}
{(q_1 w/z;q_1)_\infty}\,,
\label{cont3}\\
&\cont{\check{\Phi}^*(z)}{\check{\Phi}(w)}
=\cont{\check{\Phi}(z)}{\check{\Phi}^*(w)}
=z^{-\check{\beta}}\frac{(q_1^{(1+\check{\beta})/2}w/z;q_1)_\infty}
{(q_1^{(1-\check{\beta})/2}w/z;q_1)_\infty}\,.
\label{cont4}
\end{align}
We have also the commutation relations
\begin{align}
&\Phi(z)\Phi^*(w)=\Phi^*(w)\Phi(z) 
\Bigl(\frac{w}{z}\Bigr)^{\beta}\frac{\theta_{q_1}(q_1^{(1+\beta)/2}w/z)}
{\theta_{q_1}(q_1^{(1-\beta)/2}w/z)}\,, 
\label{PhiPhi*1}
\\
&\check{\Phi}(z)\check{\Phi}^*(w)=\check{\Phi}^*(w)\check{\Phi}(z) 
\Bigl(\frac{w}{z}\Bigr)^{\check{\beta}}
\frac{\theta_{q_1}(q_1^{(1+\check{\beta})/2}w/z)}
{\theta_{q_1}(q_1^{(1-\check{\beta})/2}w/z)}\,.
\label{PhiPhi*2}
\end{align}
The coefficients in the RHS of \eqref{PhiPhi*1}, \eqref{PhiPhi*2}
are $q_1$-periodic functions of $w/z$. 
\qed
\end{lem}

\begin{thm}\label{thm:F22}
Let $\mathcal{N}=(s_2\check{s}_2)^{M/2}(q^{-1}_2;q_1)_{M}$. 
Then the assignment 
\begin{align*}
&X^+_i(z)\mapsto \Phi(z)\check{\Phi}(q_1^iz)\,,\\ 
&X^-_i(z)\mapsto 
 \Phi^*(z)\check{\Phi}^*(q_1^{-i}z)\mathcal{N}\,,
\end{align*}
makes $\mathbb{F}_{2;2}(v;\check{v})$ a representation of 
$\mathcal{A}_{M,M-1}$ of level $(s_2,\check s_2)$.
\end{thm}

\begin{proof}
First we note that the space $\mathbb{F}_{2;2}(v;\check{v})$
is stable under the action of $X^\pm_i(z)$. 
Due to condition \eqref{integral-cond}, 
$X^\pm_i(z)$
are expanded in integer powers of $z$.
The intertwining properties of $\Phi(z)$, $\Phi^*(z)$, $\check{\Phi}(z)$, and
$\check{\Phi}^*(z)$
ensure relations (R1).

Using the contraction rules given in \eqref{cont1}--\eqref{cont4}, 
we obtain
\begin{align*}
&\cont{X^+_i(z)}{X^+_j(w)} 
={q_1^{i\cb}}\,
z^{M+1}\frac{(q_1^{1-\beta}w/z;q_1)_{j-i+M}}{(q_1w/z;q_1)_{j-i-1}}\,,
\\
&\cont{X^-_i(z)}{X^-_j(w)}
={q_1^{-i\cb}}\,
z^{M+1}
\frac{(w/z;q_1)_{i-j+1}}{(q_1^{\beta}w/z;q_1)_{i-j-M}}\,.
\end{align*}
In order to verify (R2), we may assume $i\le j$
without loss of generality.
Using \eqref{gii}, \eqref{gij1}, \eqref{gij2} and the contractions above,
one checks that
\begin{align*}
&\gamma_{i,j}(z,w)\cont{X^+_i(z)}{X^+_j(w)}=
(-1)^{j-i-1}\gamma_{j,i}(w,z)\cont{X^+_j(w)}{X^+_i(z)}\,, \\
&\gamma_{i,j}(w,z)\cont{X^-_i(z)}{X^-_j(w)}=
(-1)^{j-i-1}\gamma_{j,i}(z,w)\cont{X^-_j(w)}{X^-_i(z)}\,, 
\end{align*}
and that both sides are Laurent polynomials. 
The $X^+X^+$ and $X^-X^-$ relations follow. 

From the contractions 
\begin{align*}
&\cont{X^+_i(z)}{X^-_j(w)} 
={q_1^{-i\cb}}
z^{-M-1}\frac{(s_2q_1w/z;q_1)_{-i-j-(M+1)/2}}{(s_2^{-1}w/z;q_1)_{-i-j+(M+1)/2}}\,,
\\
&\cont{X^-_j(w)}{X^+_i(z)}
={q_1^{j\cb}}
w^{-M-1}\frac{(s_2q_1 z/w;q_1)_{i+j-(M+1)/2}}{(s_2^{-1}z/w;q_1)_{i+j+(M+1)/2}}\,,
\end{align*}
we obtain that 
\begin{align*}
\cont{X^+_i(z)}{X^-_j(w)}
=(-1)^{M+1}
\cont{X^-_j(w)}{X^+_i(z)}\,
\end{align*}
as rational functions.
Close examination shows that the poles of
$\cont{X^+_i(z)}{X^-_j(w)}$ are all simple and are located at either of the 
following lines:
\begin{align*}
(1)\quad&\quad\frac{s_2z}{q_1^rw}=1\quad 
\text{for some $r,\check{r}\ge0$, $r+\check{r}=-i-j+\frac{M-1}{2}$}\,,
\\
(2)\quad&\quad\frac{s_2w}{q_1^rz}=1\quad 
\text{for some $r,\check{r}\ge0$, $r+\check{r}=i+j+\frac{M-1}{2}$}\,.
\end{align*}
In case (1), $\Phi(z)\Phi^*(w)$ has a simple pole and 
$\check{\Phi}(q_1^i z)\check{\Phi}^*(q_1^{-j} w)$ is regular;
in case (2) it is vice versa. 

In order to show the $X^+X^-$ relation, it suffices 
 to calculate the 
residues of $X^+_i(z)X^-_j(w)$ at these lines. This can be done
by applying the following lemma.
\end{proof}

\begin{lem}\label{lem:PhiK}
For all $r\in\Z_{\ge0}$, the following identities hold on $\F_2(v)$
and on $\check{\F}_2(v)$:
\begin{align}
&\Phi^{*}(s_2z)\Phi(q_1^rz)=A\,k^-_r(z)\,,\quad 
\Phi(s_2z)\Phi^*(q_1^rz)=A\,k^+_r(z)\,,
\label{PPK1}\\
&\check{\Phi}(\cq_1^rz)\check{\Phi}^{*}(\cs_2z)=\check{A}\,\ck^-_r(z)\,,\quad
\check{\Phi}^*(\cq_1^rz)\check{\Phi}(\cs_2z)=\check{A}\,\ck^+_r(z)\,,
\label{PPK2}\\
&\mathop{\mathrm{Res}}_{w=z}\Phi(q_1^rz)\Phi^*(s_2w)\frac{dw}{w}
=-Bk_r^-(z)\,,
\quad
\mathop{\mathrm{Res}}_{w=z}\Phi^*(q_1^rz)\Phi(s_2w)
\frac{dw}{w}
=-Bk_r^+(z)\,,
\label{PPK3}\\
&\mathop{\mathrm{Res}}_{w=z}
\check{\Phi}^*(\check{s}_2z)\check{\Phi}(\check{q}_1^rw)\frac{dw}{w}
=-\check{B}\check{k}_r^-(z)\,,
\quad
\mathop{\mathrm{Res}}_{w=z}
\check{\Phi}(\check{s}_2z)\check{\Phi}^*(\check{q}_1^rw)\frac{dw}{w}
=-\check{B}\check{k}^+_r(z)\,,
\label{PPK4}
\end{align}
where 
\begin{align*}
&A=s_2^{-\beta/2}\frac{(q_1;q_1)_\infty}{(q_2^{-1};q_1)_\infty}\,,
\quad 
\check{A}=\check{s}_2^{\check{\beta}/2}
\frac{(q_1;q_1)_\infty}{(\cq_2;q_1)_\infty}\,,
\\
&B=s_2^{\beta/2}\frac{(q_1q_2;q_1)_\infty}{(q_1;q_1)_\infty}\,,\quad
\check{B} 
=\check{s}_2^{-\check{\beta}/2}
\frac{(q_1\check{q}_2^{-1};q_1)_\infty}{(q_1;q_1)_\infty}\,.
\end{align*}
\end{lem}
\begin{proof}
Using the 
free field realization of $\F_2(v)$ given in 
\eqref{E1Fock-e}, \eqref{E1Fock-f},
we find
\begin{align*}
&e^{(r)}_{q_1}(z)=
\frac{(q_1q_3;q_1)_r}{(q_1;q_1)_r} (s_3^{-1}v)^r 
:\prod_{j=0}^{r-1}\xi^{+}(q_1^jz):\,,
\\ 
&f^{(r)}_{q_1}(z)=
\frac{(q_1^{-1}q_3^{-1};q_1^{-1})_r}{(q_1^{-1};q_1^{-1})_r}
(s_3 v^{-1})^r 
:\prod_{j=0}^{r-1}\xi^{-}(q_1^{j}z):\,. 
\end{align*}
Comparing them with  definition \eqref{Phiosc}, \eqref{Phiosc*}
we obtain \eqref{PPK1}, \eqref{PPK2}. 
Formulas  \eqref{PPK3}, \eqref{PPK4} follow from these and commutation
relations \eqref{PhiPhi*1}, \eqref{PhiPhi*2}.
\end{proof}
\bigskip

\noindent{\it Example 1 (The Fock representation of quantum toroidal 
$\gl_2$ algebra).}

In the case $M=1,N=0$,  $\mathbb{F}_{2;2}(v;\check{v})$ 
is identified 
with the Fock representation of the quantum toroidal $\gl_2$ algebra $\E_2$. 
\qed
\bigskip

\subsection{Tensor product with Fock representations $\F_1(v)$, $\check\F_1(\check v)$}\label{sec:mn}

The following lemma and theorem 
are the counterparts of Lemma 5.5 and Theorem 5.6 in \cite{FJM2}.
The proofs are identical to those given in \cite{FJM2}.

\begin{lem}\label{F1 extended lemma}
The $\E_1[K]$ Fock representation $\F_1(v)$ of color $1$ has an admissible 
$\mc A_{M,1}$ module structure such that 
$X^\pm_i(z)$, $\check{e}(z)$, $\check{f}(z)$ act by zero and $\check{\psi_0},\check{C}, \check{K}$ by $1$.

Similarly, the $\check{\E}_1[\check K]$ Fock representation $\check{\F}_1(v)$ of color $1$ has an admissible $\mc A_{M,1}$ module structure such that $X^\pm_i(z)$, $e(z)$, $f(z)$ act by zero and $\psi_0,C, K$ by $1$. 
\qed
\end{lem}

\begin{thm}\label{coproduct thm}
Let $\mathcal{M}$ be an admissible representation of $\mathcal{A}_{M,N}$.

On the following tensor products, let the action of 
$\mathcal{K}_{M,N}$ be defined by \eqref{coproE1}, 
and let the generators $X^\pm_i(z)$ act as follows.\medskip

(1) ${\mc F}_{1}(v)\otimes \mathcal{M}$:
\begin{align*}
X^+_i(z)&\mapsto
k_0^-(z)\otimes X^+_{i-1/2}(s_1z)+k_1^-(q_1^{-1}z)\otimes X^+_{i+1/2}(s_1^{-1}z)
\,,\\
X^-_i(z)&\mapsto -1\otimes X^-_i(z)\,.
\end{align*}

(2) $\mathcal{M}\otimes\mc F_{1}(v)$:
\begin{align*}
X^+_i(z)&\mapsto X^+_i(z)\otimes 1\,, \\
X^-_i(z)&
\mapsto X^-_{i+1/2}(s_1z)\otimes k_0^+(z)+X^-_{i-1/2}(s_1^{-1}z)\otimes k_1^+(q_1^{-1}z)
\,.
\end{align*}

(3) $\check{\mc F}_{1}(v)\otimes \mathcal{M}$:
\begin{align*}
X^+_i(z)&\mapsto
\check{k}_0^-(q_1^iz)\otimes X^+_{i-1/2}(z)+\check{k}_1^-(q_1^{i+1}z)\otimes X^+_{i+1/2}(z)
\,,\\
X^-_i(z)&\mapsto -1\otimes X^-_i(z)\,.
\end{align*}

(4) $\mathcal{M}\otimes\check{\mc F}_{1}(v)$: 
\begin{align*}
X^+_i(z)&\mapsto X^+_i(z)\otimes 1\,, \\
X^-_i(z)&
\mapsto X^-_{i+1/2}(z)\otimes \check{k}_0^+(q_1^{-i}z)+X^-_{i-1/2}(z)\otimes \check{k}_1^+(q_1^{-i+1}z)
\,.
\end{align*}
Then (1), (3) are admissible representations of $\mathcal{A}_{M,N+1}$ 
and 
(2), (4) are admissible representations of
$\mathcal{A}^-_{M,N+1}$.
\qed
\end{thm}

Consider the special case where 
$\mathcal{M}=\mathbb{F}_{2;2}(v_2;\check{v}_2)$.
In Appendix \ref{sec:FermionicR}, Proposition \ref{prop:RPhiR}, we show 
that the R matrix of $\E_1$ on $\F_1(v_1)\otimes \F_2(v_2)$ 
transforms the intertwiner as
\begin{align*}
R^{-1}\bigl(\id\otimes \Phi(z)\bigr) R
=\frac{1}{(s_3v_1/v_2)^{-1/2}-s_3v_1/v_2}
\bigl(k_0^-(z)\otimes\Phi(s_1z)+k^-_1(q_1^{-1}z)\otimes\Phi(s_1^{-1}z)\bigr)\,.
\end{align*}
Since the  R matrix of $\check{\E}_1$ acts as identity
on $\F_1(v_1) \otimes \mathcal{M}$, we have
\begin{align*}
\mathcal{R}^{-1} 
\bigl(\id\otimes\Phi(z)\check{\Phi}(q_1^iz)\bigr)\mathcal{R}
&=R^{-1}\bigl(\id\otimes \Phi(z)\bigr) R\cdot
\check{R}^{-1}\bigl(\id\otimes\check{\Phi}(q_1^iz)\bigr)\check{R}
\\
&
\propto k_0^-(z)\otimes \Phi(s_1z)\check{\Phi}(q_1^iz)
+k_1^-(q_1^{-1}z)\otimes\Phi(s_1^{-1}z)\check{\Phi}(q_1^iz)\,.
\end{align*}
This explains the origin of the formula in Theorem \ref{coproduct thm}.
At the same time it gives a supporting evidence for  Conjecture 
\ref{conj:copro} on the coproduct.
\medskip

More generally,
starting with  $\mathbb{F}_{2;2}(v;\check{v})$
and applying Theorem \ref{coproduct thm} several times,
we obtain a family of representations  of $\mathcal{A}_{M,N}$
or $\mathcal{A}^-_{M,N}$.
For $\mathbf{v}=(v_1,\ldots,v_m)$ and $i_1,\ldots,i_m\in\{1,2,3\}$ we denote
\begin{align*}
&\F_{i_1,\ldots,i_m}(\mathbf{v})=
\F_{i_1}(v_{1})\otimes\cdots\otimes\F_{i_m}(v_{m})\,,
\\
&\mathbf{v}p^{\delta_a}=(v_1,\ldots,\overset{\underset{\smile}{a}}{v_a}p,
\ldots,v_m)\,,
\end{align*}
and similarly for the checked version. 

\begin{thm}\label{thm:F111222}
Let $m,\check{m}$ be positive integers, $1\le a\le m$, $1\le b\le \check{m}$,
and set $N=M+m+\check{m}-3$.
Then the following is an admissible representation of $\mathcal{A}_{M,N}$ 
or $\mathcal{A}^-_{M,N}$ depending on whether $m-a+\check{m}-b$
is even or odd:
\begin{align*}
\mathbb{F}_{1,\ldots,1,\overset{\underset{\smile}{a}}{2},1,\ldots,1;
1,\ldots,1,\overset{\underset{\smile}{b}}{2},1,\ldots,1}
(\mathbf{v};\check{\mathbf{v}})
=
\bigoplus_{n\in\Z}
\F_{1,\ldots,1,\overset{\underset{\smile}{a}}{2},1,\ldots,1}(\mathbf{v}q_3^{-n\delta_{a}})
\boxtimes
\check{\F}_{1,\ldots,1,\overset{\underset{\smile}{b}}{2},1,\ldots,1}(\check{\mathbf{v}}
\cq_3^{-n\delta_{b}})\,.
\end{align*}
Here we assume that
$v_a=q_1^{\lambda_a}$, $\check{v}_b=\cq_1^{\check{\lambda}_b}$, 
$\lambda_a+\check{\lambda}_b\in \Z+\frac{M+1}{2} 
$. 
\qed
\end{thm}
Representations in Theorem \ref{thm:F111222}
with different $a,b$ are related by R matrices and are 
mutually isomorphic for generic $\mathbf{v}$, $\mathbf{\check{v}}$.

\vskip1cm

\noindent{\it Example 2 (Evaluation Wakimoto
representation of quantum toroidal $\gl_2$).}
\medskip

Choose $M=-1$, $N=0$, $m=a=3$, $\check{m}=b=1$. 
We have
\begin{align*}
\check{q}_1=q_1^{-1}\,,\quad \check{q}_2=q_2q_1^2\,,\quad \check{q}_3=q_3\,,
\quad \check C=C=s_1^2s_2\,.
\end{align*}
The procedure described above endows
the
$\mathcal{K}_{-1,0}$ module
\begin{align*}
\mathbb{F}_{1,1,2;2}(v_1,v_2,v_3;\check{v}_1) 
=\bigoplus_{n\in\Z}
\bigl(\F_1(v_1)\otimes\F_1(v_2)\otimes\F_2(q_3^{-n}v_3)\bigr)
\boxtimes\check{\F}_2(\cq_3^{-n}\check{v}_1)\,
\end{align*}
with a structure of
a representation of $\mathcal{A}_{-1,0}$ 
by defining the currents
\begin{align*}
X^+_i(z)&=\Phi^{(3)}(z)\otimes  \check{\Phi}(q_1^iz)\,,\\
X^-_i(z)&=
\id^{\otimes2}\otimes\Phi^*(z)\otimes\check{\Phi}^*(q_1^{-i}z)\,,
\end{align*}
where
\begin{align*}
\Phi^{(3)}(z)
&= k_0^-(z)\otimes k^-_0(s_1z)\otimes \Phi(s_1^2z)\\
&+\Bigl(k_0^-(z)\otimes k_1^-(s_1^{-1}z)
+k_1^-(q_1^{-1}z)\otimes k_0^-(s_1^{-1}z)
\Bigr)\otimes \Phi(z)\\
&+k_1^-(s_1^{-2}z)\otimes k_1^-(s_1^{-3}z)\otimes 
\Phi(s_1^{-2}z)\,.
\end{align*}
While $X^-_i(z)$ is still a single vertex operator, $X^+_i(z)$
is now a sum of four vertex operators.

Due to \cite{FJM2}
along with the isomorphism \eqref{Mto-M},
$\mathcal{A}_{-1,0}(q_1,q_2,q_3)$ can be identified with the quantum toroidal 
$\gl_2$ algebra $\E_2(\bq_1,\bq_2,\bq_3)$
with the parameters
$\bar{q}_1=s_2$, $\bar{q}_2=s_3^2$, $\bar{q}_3=s_1^2s_2$, 
and the central element $C$ acting as the scalar $\bar{q}_3$. 
The module $\mathbb{F}_{1,1,2;2}(v_1,v_2,v_3;\check{v}_1)$
is an evaluation representation of the
Wakimoto representation of the quantum affine algebra 
$U_{\bar{s}_2}\bigl(\widehat{\gl}_2\bigr)$. 

\jm{The $q$-analog of the Wakimoto module for 
$\widehat{\mathfrak{sl}}_2$ was found in \cite{Sh}. 
This can be modified into a module over 
$\E_2(\bar q_1,\bar q_2,\bar q_3)$, by 
introducing an extra Heisenberg algebra and 
composing with the evaluation homomorphism 
$ev:\E_2(\bar q_1,\bar q_2,\bar q_3)
\to U_{\bar{s}_2}\bigl(\widehat{\gl}_2\bigr)$,
see \cite{FJM4} for more details.
To be precise, to compare with \cite{Sh}, we keep $m=3,\check{m}=1$, 
choose $a=2$ 
and consider $\mathbb{F}_{1,2,1;2}(v_1,v_2,v_3;\check{v}_1)$, 
in which $X^\pm_i(z)$ are both sums of two vertex operators. 
(As mentioned already these two representations are isomorphic
by fermionc $R$ matrix.)
}

\qed
\medskip

\noindent{\it Example 3 (Extension of deformed 
$W$ algebra of type $\gl(2n|1)$).}
\medskip

In \cite{FJM}, we considered an extension of 
the deformed $W$ algebra of type $\gl(2n|1)$ 
(enhanced by a Heisenberg algebra) by
introducing two currents $E(z),F(z)$ which carry non-trivial 
momenta. 
Let us show that this construction is a special case of the one in the present paper.

Generalizing the last remark in the
previous example, we choose $M=-1$, $m=2n+1$, $a=n+1$
$\check{m}=b=1$, and consider the $\mathcal{A}_{-1,2n-2}$ module
\begin{align*}
&\mathbb{F}_{1^n,2,1^n;2}(\mathbf{v},v,\bar{\mathbf{v}};\check{v})
=\bigoplus_{l\in\Z}
\bigl(\F_{1^n}(\mathbf{v})\otimes \F_2(q_3^{-l}v)
\otimes\F_{1^n}(\bar{\mathbf{v}})\bigr)  \boxtimes
\check{\F}_2(\cq_3^{-l}\check{v})\,,
\end{align*}
where $\mathbf{v}=(v_n,\ldots,v_1)$,
$\bar{\mathbf{v}}=(v_{\bar 1},\ldots,v_{\bar n})$.

By Theorem \ref{coproduct thm}, currents $X^\pm_{i}(z)$ have the form
\begin{align*}
&X^{+}_i(z) =Z^{+,(n)}(z)\otimes 1^{\otimes n}\otimes \check{\Phi}(q_1^iz)\,,\\
&X^{-}_{i}(z)=1^{\otimes n}\otimes Z^{-,(n)}(z)\otimes\check{\Phi}^*(q_1^{-i}z)
\cdot (-1)^n\mathcal{N}\,,
\end{align*}
where $i\in\Z$ for even $n$ and $i\in\Z+1/2$ for odd $n$.
Currents $Z^{\pm,(n)}(z)$ are sums of $2^n$ vertex operators
defined recursively by
\begin{align*}
&Z^{+,(n)}(z)=k^-_0(z)\otimes Z^{+,(n-1)}(s_1z)+k_1^{-}(q_1^{-1}z)\otimes
Z^{+,(n-1)}(s_1^{-1}z)\,,
\quad Z^{+,(0)}(z)=\Phi(z)\,,
\\ 
&Z^{-,(n)}(z)=Z^{-,(n-1)}(s_1z)\otimes k^+_0(z)+Z^{-,(n-1)}(s_1^{-1}z)\otimes
k^+_1(q_1^{-1}z)\,,
\quad Z^{-,(0)}(z)=\Phi^*(z)\,.
\end{align*}

As in \cite{FJM}, we choose the Dynkin diagram of $\gl(2n|1)$ as follows.

\begin{figure}[H]
\begin{align*}
\begin{tikzpicture}
\dynkin[root radius=.2cm, edge length=1.3cm, 
labels={n,2,1,\overline{1},\overline{2},\overline{n}}]{A}{o.otto.o}
\end{tikzpicture}
\end{align*}
\caption{The $\gl(2n|1)$ Dynkin diagram and labeling.}\label{gl fig}
\end{figure}

The currents $e(z)$, $f(z)$ 
give the  $qq$-character 
corresponding to the $(2n+1)$ dimensional representation.
\jm{This simply means that,  on the module 
$\mathbb{F}_{1^n,2,1^n;2}(\mathbf{v},v,\bar{\mathbf{v}};\check{v})$, 
they are expressed in the form 
\begin{align*}
\sum_{i=1}^n c_i M_i(z) +c_0 M_{0}(z)+\sum_{i=1}^n c_{\bar{i}}M_{\bar{i}}(z)\,,
\end{align*}
where $M_i(z)$ are vertex operators and $c_i\in\C$ are some coefficients. 
Due to the coproduct formula \eqref{coproE1} 
and the identity \eqref{psi e psi f}, 
the ratios of neighboring terms are given by the currents 
\begin{align*}
&A^{-1}_i(z)=1^{\otimes(n-i)}\otimes :\psi^-(s_1^{-1}z)e(s_1^{-1}z)^{-1}:
\otimes e(z)\otimes 1^{\otimes(n+i-1)}\,,
\\
&A^{-1}_{\bar i}(z)=1^{\otimes(n+i-1)}\otimes f(z)
\otimes :f(s_1^{-1}z)^{-1}\psi^+(s_1^{-1}z):\otimes 1^{\otimes(n-i)}\,,
\end{align*}
where $i=1,\ldots,n$, and the $e(z), f(z), \psi^\pm(z)$ 
in the right hand side stand for
the images of these currents in a single Fock module. 
This is pictured in the following figure 
where $A_{i,a}=A_i(az)$ 
(the roles of $s_1$ and $s_2$ are switched from \cite{FJM}).
}

\begin{figure}[H]
\begin{align*}
\begin{tikzpicture}
\node[rectangle,draw,minimum width = 0.75cm, 
    minimum height = 0.75cm] (r) at (-9,0) {};
 \node[rectangle,draw,minimum width = 0.75cm, 
    minimum height = 0.75cm] (r) at (-6.75,0) {};
    \node[rectangle,draw,minimum width = 0.75cm, 
    minimum height = 0.75cm] (r) at (-2.25,0) {};
    \node[rectangle,draw,minimum width = 0.75cm, 
    minimum height = 0.75cm] (r) at (0,0) {};
    \node[rectangle,draw,minimum width = 0.75cm, 
    minimum height = 0.75cm] (r) at (2.25,0) {};
    \node[rectangle,draw,minimum width = 0.75cm, 
    minimum height = 0.75cm] (r) at (6.75,0) {};
    \node at (-9,0) {\footnotesize$n$};
    \node at (-6.75,0){\footnotesize$n\hspace{-3pt}-\hspace{-3pt}1$};
    \node at (-4.5,0) {$\dots$};
    \node at (4.5,0) {$\dots$};
    \node at (-2.25,0) {\footnotesize$1$};
        \node at (0,0) {\footnotesize$0$};
        \node at (2.25,0) {\footnotesize$\overline{1}$};    
        \node at (6.75,0) {\footnotesize$\overline{n}$};  
\draw[ ->] (-8.5,0)-- node[below]{{\small $A_{n,s_1}^{-1}$}}(-7.25,0);
\draw[ ->] (-6.25,0)-- node[below]{{\small $A_{n-1,s_1^2}^{-1}$}}(-5,0);
\draw[ ->] (-4,0) -- node[below]{{\small $A_{2,s_1^{n-1}}^{-1}$}} (-2.75,0);
\draw[ ->] (-1.75,0)-- node[below]{{\small $A_{1,s_1^n}^{-1}$}}(-0.5,0);
\draw[ ->] (0.5,0)-- node[below]{\small $A^{-1}_{\overline{1},s_1^ns_2}$}(1.75,0);
\draw[ ->] (2.75,0)-- node[below]{\hspace{5pt}\small $A^{-1}_{\overline{2},s_1^{n+1}s_2}$}(4,0);
\draw[ ->] (5,0)-- node[below]{\small $A^{-1}_{\overline{n},s_1^{2n-1}s_2}\hspace{2pt}$}(6.25,0);
\end{tikzpicture}
\end{align*}
\caption{The $qq$-character corresponding to the $\gl(2n|1)$ vector representation.}\label{one box pic}
\end{figure}

The current $Z^{+,(n)}(z)$ gives the $qq$-character denoted $\xi_1^{(n)}$
in \cite{FJM}. As an illustration we take $n=3$. In the following figure,
the top left term is 
\begin{align*}
\Lambda'(z)=k_1^-(s_1^{-2}z)\otimes k^-_1(s_1^{-3}z)\otimes k^-_1(s_1^{-4}z)
\otimes \Phi(s_1^{-3}z)\otimes 1^{\otimes 3}. 
\end{align*}

\begin{figure}[H]
\begin{align*}
\begin{tikzpicture}
\node at (-6.75,0) {$\bullet$};
\node at (-4.5,0) {$\bullet$};
\node at (-2.5,0) {$\bullet$};
\node at (0,0) {$\bullet$};
\node at (-2.25,-2.25) {$\bullet$};
\node at (0,-2.25) {$\bullet$};
\node at (2.25,-2.25) {$\bullet$};
\node at (4.5,-2.25) {$\bullet$};
\draw[ ->] (-6.25,0)-- node[below]{{\small $A_{1,s_1^{-3}}^{-1}$}}(-5,0);
\draw[ ->] (-4,0)-- node[below]{{\small $A_{2,s_1^{-2}}^{-1}$}}(-2.75,0);
\draw[ ->] (-1.75,0)-- node[below]{{\small $A_{1,s_1^{-1}}^{-1}$}}(-0.5,0);
\draw[ ->] (-1.75,-2.25)-- node[below]{{\small $A_{1,s_1^{-1}}^{-1}$}}(-0.5,-2.25);
\draw[ ->] (0.5,-2.25)-- node[below]{{\small $A_{2,1}^{-1}$}}(1.75,-2.25);
\draw[ ->] (2.75,-2.25)-- node[below]{{\small $A_{1,s_1}^{-1}$}}(4,-2.25);
\draw[ ->] (-2.25,-0.5)-- node[left]{{\small $A_{3,s_1^{-1}}^{-1}$}}(-2.25,-1.75);
\draw[ ->] (0,-0.5)-- node[left]{{\small $A_{3,s_1^{-1}}^{-1}$}}(0,-1.75);
\end{tikzpicture}
\end{align*}
\caption{The $qq$-character $\xi_1^{(3)}$.}\label{char gl(2n|1)}
\end{figure}
To get the above picture we use the recursion 
$k_0^\pm(q_1z)=(\psi_0)^{\mp1}\psi^\pm(z)k^\pm_0(z)$, together with the 
identities valid on $\F_2(v)$:
\begin{align*}
e(z)=-c_2^{-1}s_3 :\Phi(q_1z)\Phi(z)^{-1}:\,,
\quad
f(z)=c_2^{-1}s_3:\Phi^*(q_1z){\Phi^*(z)}^{-1}:\,.
\end{align*}
For the current $Z^{-,(n)}(z)$, we interchange $A_i$ with $A_{\bar i}$,
and $\Lambda'(z)$ with
\begin{align*}
\bar{\Lambda}'(z)=1^{\otimes 3}\otimes
\Phi^*(s_1^{-3}z)\otimes
k_1^+(s_1^{-4}z)\otimes k^+_1(s_1^{-3}z)\otimes k^+_1(s_1^{-2}z)\,.
\end{align*}

Fix $i$, and set $E(z)=X^+_i(z)$, $F(z)=X^-_{-i}(z)$,
$\Lambda(z)=\Lambda'(z)\otimes\check{\Phi}(q_1^iz)$, 
$\bar{\Lambda}(z)=\bar{\Lambda}'(z)\otimes\check{\Phi}^*(q_1^iz)$.
It is straightforward to verify the contractions
\begin{align*}
&\cont{\Lambda(z)}{\Lambda(w)}=\frac{z-w}{z-s_3^2w}\,,\quad \cont{\bar\Lambda(z)}{\bar\Lambda(w)}=\frac{z-w}{z-s_3^{-2}w}\,,
\\
&\cont{\Lambda(z)}{\bar\Lambda(w)}=\cont{\bar\Lambda(z)}{\Lambda(w)}=1\,.\nn
\end{align*}
From the relations in $\mathcal{A}_{-1,2n-2}$, we obtain 
\begin{align*}
&(z-q_3w)E(z)E(w)+(w-q_3z)E(w)E(z)=0\,,\\
&(w-q_3z)F(z)F(w)+(z-q_3w)F(w)F(z)=0\,,\\
&[E(z),F(w)]=\sum_{r=0}^{n-1}
\delta\Bigl(q_1^{n-r}s_2\frac{w}{z}\Bigr)
k^+_r(q_1^{-r}w)\check{k}^+_{n-1-r}(q_1^{n-1-r+i}w)
\\
&\hspace{2.3cm}-\sum_{r=0}^{n-1}\delta\Bigl(q_1^{n-r}s_2\frac{z}{w}\Bigr)
k^-_r(q_1^{-r}z)\check{k}^-_{n-1-r}(q_1^{n-1-r+i}z)\,.
\end{align*}
Thus we recover the relations in Theorem 3.4
and Theorem 3.6 in \cite{FJM} wherein the role of $s_1$ and $s_2$ are switched.

\qed
\medskip

\subsection{Representations $\mathbb{F}_{2,2;2,2}$}\label{sec:s22s22}

We now consider the ``tensor product'' of the representations 
$\mathbb{F}_{2;2}$ given in Section \ref{sec:s2s2}.
In this section we assume that $M$ is odd. 

The coproduct conjecture in Section \ref{sec:corpo} says that 
this should be a representation of $\mathcal{A}_{M,2M-2}$ on the space 
\begin{align*}
\mathbb{F}_{2,2;2,2}(v_1,v_2;\check{v}_1,\check{v}_2) 
=\bigoplus_{n_1,n_2\in\Z}
\bigl(\F_2(q_3^{-n_1}v_1)\otimes \F_2(q_3^{-n_2}v_2)\bigr)
\boxtimes
\bigl(\check{\F}_2(\cq_3^{-n_1}\check{v}_1)\otimes 
\check{\F}_2(\cq_3^{-n_2}\check{v}_2)\bigr)\,.
\end{align*}
Here we set $v_i=q_1^{\lambda_i}$, $\check{v}_i=\cq_1^{\check{\lambda}_i}$, 
assuming
\begin{align}
\lambda_i+\check{\lambda}_i\in \Z
\,,\quad i=1,2\,.
\label{F2222lambda}
\end{align}
We set
\begin{align*}
&L=\lambda_1+\check{\lambda}_1-\lambda_2-\check{\lambda}_2\in\Z\,.
\end{align*}
We shall write $P_1=P\otimes 1$, $P_2=1\otimes P$, $P_{1,2}=P_1-P_2$, 
and similarly for $\check{P}_{1,2}$. 
The eigenvalues of $P_i$, $\check{P}_i$ 
are $\lambda_i+n_i\beta$ and $\check{\lambda}_i+n_i\check{\beta}$, respectively.
In particular
\begin{align}
&P_{1,2}+\check{P}_{1,2}=L+(n_1-n_2)(M+1)\,.
\label{P12P12}
\end{align}

According to the conjecture the currents $X^\pm_i(z)$ should act as
\begin{align*}
&\Delta_{M,M-1} X^+_i(z) 
=\Phi(z)\check{\Phi}(q_1^iz)\otimes 1+ 
R^{-1}(1\otimes\Phi(z))R
\cdot \check{R}^{-1}(1\otimes\check{\Phi}(q_1^iz))
\check{R} 
\,,\\
&\Delta_{M,M-1} X^-_i(z) 
=R^{-1}(\Phi^*(z)\otimes1)R
\cdot \check{R}^{-1}(\check{\Phi}^*(q_1^{-i}z)\otimes1)\check{R} 
+1\otimes \Phi^*(z)\check{\Phi}^*(q_1^{-i}z)
\,.
\end{align*}
We shall construct the representations 
replacing the objects $R^{-1}(1\otimes\Phi(z))R$, etc. by
screened intertwiners, which have
the same intertwining properties;
see \eqref{level2-E}, \eqref{level2-F} below.

From now on, we use the symbol
\begin{align*}
[u]=q_1^{u(u-1)/2}\theta_{q_1}(q_1^u)\,,
\quad 
\end{align*}
which satisfies
\begin{align*}
[u+1]=-[u]=[-u]\,.
\end{align*}
Set
\begin{align*}
f(z;\beta,P)&=\frac{[u+P+\frac{1-\beta}{2}]}{[u+\frac{1+\beta}{2}]}\,
\quad \text{for $z=q_1^u$}.
\end{align*}
We introduce the screened intertwiners by
\begin{align}
&\Phi_+(z)=\Phi(z)\otimes \id\,,\quad 
\Phi_-(z)=\int_\Gamma \frac{dx}{2\pi \sqrt{-1} x}\Phi_+(z)S(x)f(x/z;\beta,P_{1,2})\,,
\label{scv}
\\
&\Phi^*_+(z)=\id\otimes \Phi^*(z)\,,\quad 
\Phi^*_-(z)=\int_\Gamma \frac{dx}{2\pi \sqrt{-1} x}
\Phi^*_+(z)S(x)f(x/z;\beta,P_{1,2})\,.
\label{scv*}
\end{align}
Here 
\begin{align*}
S(x)=\Phi^*(x)\otimes\Phi(x)
\end{align*}
is the screening current.\footnote{The symbol $S(x)$ should
not be confused with the antipode.} 
The set of poles of the integrands of \eqref{scv},\eqref{scv*}
consists of  two sets,
\begin{align}
\{s_2^{-1}q_1^nz\mid n\ge0\}\,,\quad \{s_2q_1^{-n}z\mid n\ge0\}\,, 
\label{sc-poles}
\end{align}
and the contour $\Gamma$ is a simple closed curve which separate them.
\medskip

\begin{rem}\label{rem:screened VO}
{\rm To motivate the definition \eqref{scv}--\eqref{scv*}, 
note that the formal Jackson integral
\begin{align*}
\mathcal{S}(\xi)&
=\int_0^{\xi\infty}S(x)\frac{d_{q_1}x}{x}
=(1-q_1)\sum_{n=-\infty}^\infty\Phi^*(q_1^n\xi)\otimes\Phi(q_1^n\xi)\,
\end{align*}
gives the composition of intertwiners
\begin{align*}
&\F_2(q_3^{-1}v_1)\otimes\F_2(v_2)\longrightarrow
\F_2(v_1)\otimes V_1(\xi)\otimes \F_2(v_2)
\longrightarrow
\F_2(v_1)\otimes\F_2(q_3^{-1}v_2)\,.
\end{align*}
The above expression for $\mathcal{S}(\xi)$ is divergent. 
Nevertheless, the composition 
\begin{align*}
\Phi_+(z)\circ\mathcal{S}(\xi)
=\int_0^{\xi\infty}\Phi_+(z) S(x)\frac{d_{q_1}x}{x}
\end{align*}
is convergent and yields a well-defined homomorphism of $\E_1$ modules
\begin{align*}
&V_1(z)\otimes \F_2(q_3^{-1}v_1)\otimes\F_2(v_2)
\longrightarrow 
\F_2(q_3^{-1}v_1)\otimes\F_2(q_3^{-1}v_2)\,.
\end{align*}
Different choices of $\xi$ give the same operator up to proportionality. 
The contour integral \eqref{scv} is so designed that
the sum of residues over the sets of poles  \eqref{sc-poles}
yields (up to proportionality)
the Jackson integral with $\xi=s_2^{-1}z$ or $\xi=s_2z$, 
respectively.

As is well known \cite{FHHSY}, the $e(z)$ current of $\E_1$ on 
$\F_2(v_1)\otimes\F_2(v_2)$ gives the generating current of
the deformed Virasoro algebra enhanced
by a Heisenberg algebra. 
The screened intertwiner \eqref{scv} first appeared in \cite{L,LP}
in the context of the deformed Virasoro algebra, i.e. in the absence of the 
extra Heisenberg algebra.
\qed}
\end{rem}
\medskip

For the construction of representations, we need also
the $\check{\E}_1[\check K]$ counterparts 
$\check{\Phi}_\pm(z)$, $\check{\Phi}^*_\pm(z)$. They
are defined by \eqref{scv}, \eqref{scv*} with the replacement
\begin{align*}
\Phi_\pm(z)\to\check{\Phi}_\pm(z)\,,\quad
\Phi^*_\pm(z)\to\check{\Phi}^*_\pm(z)\,,\quad
f(z;\beta,P_{1,2})\to f(z;\check{\beta},\check{P}_{1,2})\,.
\end{align*}

We slightly change the normalization
\begin{align}
&\tilde{\Phi}_\pm(z)=\pm \Phi_\pm(z)\frac{1}{[P_{1,2}]}\,,\quad
\tilde{\Phi}^*_\pm(z)=\mp \Phi_\pm^*(z)\frac{1}{[P_{1,2}]}\,.
\label{Phitilde2}
\end{align}
This has the effect of transposing the $R$ matrix in 
\eqref{PPR}--\eqref{PP*R}.

Define
\begin{align}
&X^+_i(z)=
\tilde{\Phi}_+(z)\check{\Phi}_+(q_1^iz)
+\tilde{\Phi}_-(z)\check{\Phi}_-(q_1^iz)
\,,
\label{level2-E}
\\
&\mathcal{N}_2^{-1}X^-_i(z)
=
\tilde{\Phi}^*_+(z)\check{\Phi}^*_+(q_1^{-i}z)
+\tilde{\Phi}^*_-(z)\check{\Phi}^{*}_-(q_1^{-i}z)
\,,
\label{level2-F}
\end{align}
where $\mathcal{N}_2$ is a normalization constant.

Recall that
\begin{align*}
L=\lambda_1+\check{\lambda}_1-\lambda_2- \check{\lambda}_2 
\end{align*}
is an integer due to \eqref{F2222lambda}.

The main result of this section is the following.
\begin{thm}\label{thm:level2}
With the choice $\mathcal{N}_2^{-1}=(-1)^L AB\check{B}^2/[\beta]^2$, 
the operators  \eqref{level2-E}, \eqref{level2-F}
yield an admissible representation of 
$\mathcal{A}_{M,2M-2}$ on 
$\mathbb{F}_{2,2;2,2}(v_1,v_2;\check{v}_1,\check{v}_2) $.
\qed
\end{thm}

We give a proof of Theorem \ref{thm:level2} in Appendix \ref{sec:Proof}.

\medskip

It is possible to extend the 
representation $\mathbb{F}_{2,2;2,2}$
to the case involving \jm{$m$} % $n$ 
copies of Fock spaces
\begin{align*}
&\mathbb{F}_{2,\ldots,2;2,\ldots,2}(v_1,\ldots,v_m;
\check{v}_1,\ldots,\check{v}_m) \\
&=\bigoplus_{n_1,\ldots,n_m\in\Z}
\bigl(\F_2(q_3^{-n_1}v_1)\otimes \cdots\otimes
\F_2(q_3^{-n_m}v_m)\bigr)
\boxtimes
\bigl(\check{\F}_2(\cq_3^{-n_1}\check{v}_1)\otimes \cdots\otimes
\check{\F}_2(\cq_3^{-n_m}\check{v}_m)\bigr)\,.
\end{align*}
This is a representation of $\mathcal{A}_{M,m(M-1)}$. 
In Appendix \ref{sec:2222} we present  
formulas which are relevant for this computation.
\medskip

In order to make contact with Conjecture \ref{conj:copro}, 
let us rewrite the screened intertwiner \eqref{scv} by 
taking residues at \eqref{sc-poles}. 
Using the identities in Lemma \ref{lem:PhiK} we find
\begin{align*}
\Phi_-(z)
&\propto \sum_{r\ge0}k^+_r(s_2^{-1}z)\otimes \Phi(s_2^{-1}q_1^rz)\\ 
&\propto \sum_{r\ge0}k^-_r(s_2^{-1}z)\otimes \Phi(s_2q_1^{-r}z)\,,
\end{align*}
and similar formulas for $\check{\Phi}_-(z)$. 
Therefore, up to proportionality constants,  one can rewrite
the second term of \eqref{level2-E} in either of the following forms:
\begin{align*}
\Phi_-(z)\check\Phi_-(q_1^iz)
&\propto \sum_{r,\check r\ge0}
k^+_r(s_2^{-1}z)\check k^-_{\check r}(\check s_2^{-1}q_1^iz)
\otimes 
\Phi(s_2^{-1}q_1^{r}z)\check\Phi(\check s_2^{-1}\check q_1^{\check r}q_1^iz)
\\
&\propto \sum_{r,\check r\ge0}
k^-_r(q_1^{-r}z)\check k^-_{\check r}(\check q_1^{-\check r}q_1^iz)
\otimes 
\Phi(s_2q_1^{-r}z)\check\Phi(\check s_2\check q_1^{-\check r}q_1^iz)\,.
\end{align*}
This leads us to speculate that 
the conjectured coproduct may be written explicitly as infinite sums:
\begin{align}
&\Delta_{N_1,N_2}
X^+_i(z)=X^+_i(z)\otimes 1+\sum_{r,\check r\ge0} k_{r}^-(q_1^{-r}z)
\check k^-_{\check r}(\cq_1^{-\check r}q_1^iz)\otimes 
X^+_{r+\check r+i-N_1/2}(q_1^{-r}C_1z)\,,
\label{Delta-inf}\\
&\Delta_{N_1,N_2}
X^-_i(z)=\sum_{r,\check r\ge0} 
X^-_{-r-\check r+i+N_2/2}(q_1^{-r}C_2z)
\otimes k_r^+(q_1^{-r}z)\check k_{\check r}^+(\cq_1^{-\check r}q_1^{-i}z)
+1\otimes X^-_i(z)\,.\nn
\end{align} 
Writing these formulas, we have adjusted 
the constant in front of the sum 
so that the coassociativity holds true. 
In the cases treated in Theorem \ref{coproduct thm} 
the series truncate,  
the only non-vanishing terms being $r=0,1$ and $\check r=0$. 
In general, however, we do not know when such formal sums make sense.
We only note that 
the above formula and $\Delta k^\pm_r(z)$ in Lemma \ref{lem:DeltaK}
share a close similarity with the Drinfeld coproduct for 
$U_{s_2}\widehat{\mathfrak{sl}}_2$,
\begin{align*}
&\Delta X^+_i= X^+_i\otimes 1+\sum_{r\ge0}(C^{-i-r}\Psi^-_{-r})
\otimes X^+_{r+i}\,,\\
&\Delta X^-_i=\sum_{r\ge0} X^-_{-r+i}\otimes (C^{-i+r}\Psi^+_r)
+1\otimes X^-_i\,,\\
&\Delta \Psi^+_r=\sum_{r_1+r_2=r\atop r_1,r_2\ge0}
\Psi^+_{r_1}\otimes(C^{-r_1}\Psi^+_{r_2})\,, \\
&\Delta \Psi^-_{-r}=\sum_{r_1+r_2=r\atop r_1,r_2\ge0}
(C^{r_2}\Psi^-_{-r_1})\otimes\Psi^-_{-r_2}\,,
\end{align*}
where $X^\pm(z)=\sum_{i\in\Z}X^\pm_iz^{-i}$,
$\Psi^\pm(z)=\sum_{r\ge0}\Psi^{\pm}_{\pm r}z^{\mp r}$
are the Drinfeld currents.

\appendix

\section{Fermionic R matrix}\label{sec:FermionicR}

In this section we describe the R matrix for the tensor product
$\F_1(v_1)\otimes \F_2(v_2)$. 
We set $v_2/v_1=q_3^\gamma$, and assume that $2\gamma\not\in \Z$. 
We shall work with the direct sum
\begin{align*}
\Fb=\bigoplus_{l\in\Z}\Fb_l\,, 
\quad \Fb_l=\F_1(s_3^{-l}v_1)\otimes\F_2(s_3^l v_2)\,.
\end{align*}
Let $\ket{\emptyset}_{\Fb_l}=\ket{\emptyset}_{\F_1(s_3^{-l}v_1)}
\otimes\ket{\emptyset}_{\F_2(s_3^lv_2)}$.

On each $\Fb_l$ we set
$h_{\pm r}^{(1)}=h_{\pm r}\otimes 1$,
$h^{(2)}_{\pm r}=1\otimes h_{\pm r}$ and 
\begin{align*}
&a_{-r}=\Delta h_{-r}=h^{(1)}_{-r}+s_1^r h^{(2)}_{-r}\,,\\ 
&a_{r}=\Delta h_{r}=s_2^{-r}h^{(1)}_{r}+ h^{(2)}_{r}\,,\\ 
&b_{-r}=(1-q_2^r)s_1^{r}s_2^{-r} h^{(1)}_{-r}-(1-q_1^r)s_2^{-r}h^{(2)}_{-r}\,,\\
&b_r=-(1-q_2^r)s_2^{-r}h^{(1)}_r-(1-q_1^{-r})h^{(2)}_r\,.
\end{align*}
Then 
\begin{align*}
[a_r,b_s]=0\,,\quad [b_r,b_{s}]=-\delta_{r+s,0}\frac{1}{r}\,.
\end{align*}
Define further the zero mode operators 
$e^{Q_b}:\Fb_l\to\Fb_{l+1} $ and $b_0:\Fb_l\to\Fb_l$ which commute with 
$a_r,b_r$, $r\neq 0$, and 
\begin{align*}
e^{Q_b}\ket{\emptyset}_{\Fb_l}=\ket{\emptyset}_{\Fb_{l+1}}\,,\quad
b_0\bigl|_{\Fb_l}=(\gamma+l)\cdot\id_{\Fb_l}\,.
\end{align*}
We have $[b_0,Q_b]=1$. 

Introduce the screening current
$\Sc(z)$ and its conjugate $\Sc^*(z)$ by
\begin{align*}
&\Sc(z)=e^{Q_b}z^{b_0+1/2}\Sc^{osc}(z)\,,
\quad \Sc^{osc}(z)=:\exp\bigl(\sum_{r\neq0}b_rz^{-r}\bigr):\,,\\
&\Sc^*(z)=e^{-Q_b}z^{-b_0+1/2}\Sc^{*osc}(z)\,,
\quad \Sc^{*osc}(z)=:\exp\bigl(-\sum_{r\neq0}b_rz^{-r}\bigr):\,.
\end{align*}
Their Fourier modes $\Sc_n,\Sc^*_n$ defined by 
\begin{align*}
\Sc(z)=\sum_{n\in\Z}\Sc_n z^{-n+\gamma-1/2}\,,
\quad 
\Sc^*(z)=\sum_{n\in\Z}\Sc^*_n z^{-n-\gamma+1/2}\,,
\end{align*}
are ordinary fermions satisfying
\begin{align*}
[\Sc_m,\Sc_n]_+=[\Sc^*_m,\Sc^*_n]_+=0\,,\quad
[\Sc_m,\Sc^*_n]_+=\delta_{m+n,0}\,.
\end{align*}
\medskip

\begin{lem}\label{lem:efFerm}
The currents of $\E_1(q_1,q_2,q_3)$ acting on $\Fb_l$
are expressed as
\begin{align*}
&(s_3-s_3^{-1})\Delta e(z) =s_2^{-2b_0-1}s_3^{-l}v_1\,
\Sc(s_1^{-1}s_2z)\Bigl(\Sc^*(s_3z)-\Sc^*(s_3^{-1}z)\Bigr)A(z)\,,\\
&(s_3-s_3^{-1})\Delta f(z) =s_2^{2b_0+1}s_3^lv_1^{-1}\,
\Sc(z)\Bigl(\Sc^*(q_2z)-\Sc^*(q_1^{-1}z)\Bigr)B(z)\,,
\end{align*}
where
\begin{align}
&A(z)=\exp\Bigl(-\sum_{r>0}(1-q_1^r)(1-q_2^r)q_3^ra_{-r}z^r\Bigr)
\exp\Bigl(-\sum_{r>0}(1-q_1^r)(1-q_2^r)s_3^r a_rz^{-r}\Bigr) \,,
\label{AA}
\\
&B(z)=\exp\Bigl(\sum_{r>0}(1-q_1^r)(1-q_2^r)s_3^ra_{-r}z^r\Bigr)
\exp\Bigl(\sum_{r>0}(1-q_1^r)(1-q_2^r) a_rz^{-r}\Bigr) \,.
\label{BB}
\end{align} 
\end{lem}
\begin{proof}
This can be shown by a direct computation.  
\end{proof}

We consider alongside the tensor products in the opposite order
\begin{align*}
\Fbb=\bigoplus_{l\in\Z}\Fbb_l\,, 
\quad \Fbb_l=\F_2(s_3^l v_2)\otimes\F_1(s_3^{-l}v_1)\,.
\end{align*}
We define the corresponding objects
$\bar a_r,\bar b_r$, $\bar \Sc(z)$, $\bar \Sc^*(z)$.
They are obtained from the formulas above 
by the replacement $q_1\leftrightarrow q_2$,
$v_1\leftrightarrow v_2$ and $\gamma\leftrightarrow -\gamma$.

\begin{prop}
There exists a unique  map $\Rv
:\Fb\to\Fbb$
satisfying the intertwining properties
\begin{align}
 &\Rv\,a_r \Rv^{-1}=\bar a_r\,,
\label{Rar}
\\
&(q_1^{b_0-1/2}-q_2^{-b_0+1/2})
\Rv\,\Sc(z)\Rv^{-1}
=\bar{\Sc}^{*}(q_1z)-\bar{\Sc}^{*}(q_2^{-1}z)\,,
\label{RSz}
\\
&\Rv
\bigl({\Sc}^{*}(q_1^{-1}z)-{\Sc}^{*}(q_2z)\bigr)
\Rv^{-1}
=
\bar{\Sc}(z)
(q_1^{b_0-1/2}-q_2^{-b_0+1/2})\,,
\label{RS*z}
\end{align}
and the normalization $\Rv\ket{\emptyset}_{\Fb_0}
=\ket{\emptyset}_{\Fb_0}$.

For each $l\in\Z$, $\Rv:\Fb_l\to\bar\Fb_l$ 
gives an intertwiner of $\E_1(q_1,q_2,q_3)$ modules.
\end{prop}
\begin{proof}
 In terms of Fourier modes the intertwining properties 
\eqref{RSz} and \eqref{RS*z}
read
\begin{align*}
&\Rv\Sc_n \Rv^{-1}=(q_1^{b_0-1/2}-q_2^{-b_0+1/2})^{-1}
(q_1^{-n+\gamma-1/2}-q_2^{n-\gamma+1/2})
\cdot \bar{\Sc}^*_n\,,
\\
&\Rv\Sc^*_n \Rv^{-1}=
\bar{\Sc}_n\cdot
(q_1^{b_0-1/2}-q_2^{-b_0+1/2})
(q_1^{-n+\gamma-1/2}-q_2^{n-\gamma+1/2})^{-1}\,.
\end{align*}
Since the right hand sides satisfy the same anti-commutation relations
as do $\{\Sc_n,\Sc^*_n\}$, 
such an $\Rv$ exists. 
It is unique because $\Fb$ is irreducible under the action of 
the fermions  $\{\Sc_n,\Sc^*_n\}$ and the Heisenberg algebra 
generated by $\{a_r\}$. 
Combining \eqref{Rar}--\eqref{RS*z} we obtain
\begin{align*}
&\Rv\cdot
\Sc(s_1^{-1}s_2z)\Bigl(\Sc^*(s_3z)-\Sc^*(s_3^{-1}z)\Bigr)A(z)
\cdot \Rv^{-1} 
=\bar \Sc(z)
\bigl(\bar \Sc^*(s_3z)-\bar \Sc^*(s_3^{-1}z)\bigr)
\bar A(z)\,,
\\
&\Rv\cdot
\Sc(z)\bigl(\Sc^*(q_2z)-\Sc^*(q_1^{-1}z)\bigr)B(z)
\cdot \Rv^{-1} 
=\bar \Sc(z)
\bigl(\bar{\Sc}^*(q_1z)-\bar{\Sc}^*(q_2^{-1}z)\bigr)
\bar{B}(z)\,,
\end{align*}
where $\bar{A}(z),\bar{B}(z)$ are obtained from \eqref{AA}, \eqref{BB} 
replacing $a_r$ by $\bar a_r$.
Here we used also the anti-commutativity of $\bar \Sc(z)$, $\bar{\Sc}^*(w)$. 
Comparing this with Lemma \ref{lem:efFerm}, 
we conclude that $\Rv$ is an intertwiner of $\E_1$ modules. 
\end{proof}

The intertwiners $\Phi(z),\Phi^*(z)$ are related to the screening
currents as follows. 

On $\F_1(v_1)\otimes\F_2(v_2)$, we have
\begin{align}
&1\otimes\Phi^{*osc}(z)=\Sc^{osc}(q_1^{-1}z)T(z)\,,
\label{PhiS}
\\
&T(z)=\exp\Bigl(-\sum_{r>0}(1-q_2^r)s_3^r a_{-r}z^r\Bigr) 
\exp\Bigl(\sum_{r>0}(1-q_2^r)q_1^r a_rz^{-r}\Bigr)\,.
\nn
\end{align}
On $\F_2(v_2)\otimes\F_1(v_1)$, we have
\begin{align}
&\Phi^{osc}(z)\otimes 1
=\bar{\Sc}^{osc}(s_3z)\bar{T}^*(z)\,,
\label{Phi*S}
\\
&\bar{T}^{*}(z)=
\exp\Bigl(\sum_{r>0}(1-q_2^r)q_3^r \bar a_{-r}z^r\Bigr) 
\exp\Bigl(-\sum_{r>0}(1-q_2^r)s_1^rs_2^{-r}\bar a_rz^{-r}\Bigr)\,.
\nn
\end{align}

Define 
$\bar T(z)=\Rv T(z)\Rv^{-1}$, 
$T^*(z)=\Rv^{-1}\bar T^*(z)\Rv$. 
Inserting \eqref{PhiS}, \eqref{Phi*S} into the intertwining relations
\eqref{RSz}, \eqref{RS*z}, 
we obtain 
\begin{align*}
&\Rv(1\otimes \Phi^{*osc}(z))\Rv^{-1}
=\frac{1}{1-s_3v_2/v_1}
\Bigl(\bar{\Sc}^{*osc}(z)-\frac{s_3v_2}{v_1}
\bar{\Sc}^{*osc}(q_3z)\Bigr)\bar{T}(z)\,,
\\
&\Rv^{-1}(\Phi^{osc}(z)\otimes 1)\Rv
=
\frac{1}{1-s_3v_1/v_2}
\Bigl({\Sc}^{*osc}(q_2s_3z)-\frac{s_3v_1}{v_2}
{\Sc}^{*osc}(q_1^{-1}s_3z)\Bigr)T^*(z)\,.
\end{align*}

Further, a direct computation shows that
\begin{align*}
&\bar{\Sc}^{*osc}(z)\bar{T}(z)
=\Phi^{*osc}(s_1z)\otimes K^{-1}k_0^{+}(z)\,,\\
&\bar{\Sc}^{*osc}(q_3z)\bar{T}(z)
=\Phi^{*osc}(s_1^{-1}z)\otimes Kk_0^{-}(s_1^{-1}z)\,,\\
&\Sc^{*osc}(q_2s_3z)T^*(z)=Kk^{-}_0(z)\otimes\Phi^{osc}(s_1z)\,,\\ 
&\Sc^{*osc}(q_1^{-1}s_3z)T^*(z)=K^{-1}k^{+}_0(s_1^{-1}z)\otimes\Phi^{osc}(s_1^{-1}z)\,.
\end{align*}

Restoring zero mode, and using Lemma \ref{k on fock lem}, we arrive at 
\begin{prop}\label{prop:RPhiR}
The following relations hold on  
$\F_2(v_2)\otimes\F_1(v_1)$ or $\F_1(v_1)\otimes\F_2(v_2)$ respectively:
\begin{align*}
&\Rv(1\otimes\Phi^*(z))\Rv^{-1}
=\frac{1}{(s_3v_2/v_1)^{-1/2}-(s_3v_2/v_1)^{1/2}} 
\Bigl(\Phi^*(s_1z)\otimes k_0^+(z)+\Phi^{*}(s_1^{-1}z)\otimes
k_1^+(s_1^{-2}z)
\Bigr)\,,
\\
&
\Rv^{-1}(\Phi(z)\otimes 1)\Rv=
\frac{1}{(s_3v_1/v_2)^{-1/2}-(s_3v_1/v_2)^{1/2}} 
\Bigl(k_0^-(z)\otimes\Phi(s_1z)
+k_1^-(s_1^{-2}z)\otimes \Phi(s_1^{-1}z)\Bigr)\,.
\end{align*}
\qed
\end{prop}

\begin{rem}
{\rm Let $R_{12}$ (resp. $R_{21}$)
be the image of the universal R matrix of $\E_1$
on $\F_1(v_1)\otimes\F_2(v_2)$ (resp.  $\F_2(v_2)\otimes\F_1(v_1)$).
Since these tensor products are irreducible, we have
$\Rv=const. \sigma R_{12}$ (resp. $\Rv^{-1}=const. \sigma R_{21}$),
where $\sigma$ stands for the permutation $\sigma w_1\otimes w_2=w_2\otimes w_1$. 
Hence the left hand sides of the formulas in 
Proposition \ref{prop:RPhiR} can be written as
\begin{align*}
&\Rv(1\otimes\Phi^*(z))\Rv^{-1}=R_{21}^{-1}(\Phi^*(z)\otimes1)R_{21}\,,
\\
&\Rv^{-1}(\Phi(z)\otimes 1)\Rv=R_{12}^{-1}(1 \otimes \Phi(z))R_{12}\,.
\end{align*}
\qed
}\end{rem}

\section{Proof of Theorem \ref{thm:level2}
}\label{sec:Proof}

In this section we give a proof of Theorem  \ref{thm:level2}. 
We assume $M$ is odd, so that $X^\pm_i(z)$ are labeled by $i\in\Z$.
\medskip

\subsection{Properties of screened  intertwiners}\quad
We begin by studying the commutation relations of screened 
intertwiners.

Let
\begin{align*}
R(u;\beta,P)=
\begin{pmatrix}
1 & 0 & 0 &0 \\ 
0&\frac{[u+P]}{[u+\beta]}\frac{[\beta]}{[P]}
&\frac{[u]}{[u+\beta]}\frac{[\beta+P]} {[P]}&0\\
0&\frac{[u]}{[u+\beta]}\frac{[\beta-P]} {[-P]}
&\frac{[u-P]}{[u+\beta]}\frac{[\beta]}{[-P]}
&0\\
0&0&0&1\\
\end{pmatrix} \,.
\end{align*}
be the elliptic dynamical $R$ matrix.

\begin{prop}\label{commPhi}
Let $R(z)=R(u;\beta,P_{1,2})$ with $z=q_1^u$.
The operators ${\Phi}_\pm(z)$, ${\Phi}^*_\pm(z)$
satisfy the following commutation relations as meromorphic functions:
\begin{align}
&\Phi_{\epsilon_1}(z) \Phi_{\epsilon_2}(w) 
=\rho(z/w)
\sum_{\epsilon_1',\epsilon_2'=\pm}
\Phi_{\epsilon'_1}(w) \Phi_{\epsilon'_2}(z) 
R(z/w)_{\epsilon'_1,\epsilon_2';\epsilon_1,\epsilon_2}\,,
\label{PPR}
\\
&\Phi^*_{\epsilon_1}(z)\Phi^*_{\epsilon_2}(w) 
=\rho^*(z/w)
\sum_{\epsilon_1',\epsilon_2'=\pm}
\Phi^*_{\epsilon'_1}(w) \Phi^*_{\epsilon'_2}(z) 
R(z/w)_{\epsilon'_1,\epsilon_2';\epsilon_1,\epsilon_2}\,,
\label{P*P*R}
\\
&\Phi_{\epsilon_1}(z)\Phi^*_{\epsilon_2}(w) 
=
\sum_{\epsilon_1',\epsilon_2'=\pm}
\Phi^*_{\epsilon'_1}(w) \Phi_{\epsilon'_2}(z) 
R(z/w)_{\epsilon'_1,\epsilon_2';\epsilon_1,\epsilon_2}\,,
\label{PP*R}
\end{align} 
where
\begin{align}
&\rho(z)=z^{\beta}
\frac{(q_1z,q_1^{1-\beta}z^{-1};q_1)_\infty} 
{(q_1^{1-\beta}z, q_1 z^{-1};q_1)_\infty}\,,\quad
\rho^*(z)=z^{\beta}
\frac{(q_1^{\beta}z,z^{-1};q_1)_\infty} 
{(z, q_1^\beta z^{-1};q_1)_\infty}\,.
\label{rho}
\end{align}

Similarly, the same commutation relations hold wherein 
$\Phi_\pm(z)$, $\Phi^*_\pm(z)$
are replaced by  $\check{\Phi}_\pm(z)$, $\check{\Phi}^*_\pm(z)$,
$R(z)$ by $R(u;\check{\beta},\check{P}_{1,2})$, and 
\eqref{rho} by
\begin{align}
&\check{\rho}(z)=z^{\check{\beta}} 
\frac{(q_1^{\check{\beta}}z,z^{-1};q_1)_\infty} 
{(z, q_1^{\check{\beta}} z^{-1};q_1)_\infty}\,,\quad
\check{\rho}^*(z)=z^{\check{\beta}}
\frac{(q_1z,q_1^{1-\check{\beta}}z^{-1};q_1)_\infty} 
{(q_1^{1-\check{\beta}}z, q_1z^{-1};q_1)_\infty}\,.
\label{rhoc}
\end{align}
\end{prop}
\begin{proof}
For the deformed Virasoro algebra,
this result was obtained in \cite{LP}.
The proof given there works in the present setting equally well.
\end{proof}

In addition we need information about the location of poles.
\begin{lem}\label{lem:location}
The operators
\begin{align}
&\Bigl(\frac{w}{z}\Bigr)^{-\beta}
\frac{(q_1w/z;q_1)_\infty}{(q_1^{1-\beta}w/z;q_1)_\infty}
\Phi_{\epsilon_1}(z)\Phi_{\epsilon_2}(w)
\,,
\label{PhiPhi}\\
&\Bigl(\frac{w}{z}\Bigr)^{-\beta}
\frac{(q_1^{\beta}w/z;q_1)_\infty}{(w/z;q_1)_\infty}
\Phi^*_{\epsilon_1}(z)\Phi^*_{\epsilon_2}(w)\,,
\label{Phi*Phi*}\\
&\Phi_{\epsilon_1}(z)\Phi^*_{\epsilon_2}(w)\,,
\quad \Phi^*_{\epsilon_1}(z)\Phi_{\epsilon_2}(w)\,,
\label{PhiPhi*}
\end{align} 
have poles only when 
$\epsilon_1+\epsilon_2=0$ and at $w/z=q_2 q_1^{-r}$, $r\ge0$.
All poles are simple.

Similarly, the operators
\begin{align}
&\Bigl(\frac{w}{z}\Bigr)^{-\check{\beta}}
\frac{(q_1^{\check{\beta}}w/z;q_1)_\infty}{(w/z;q_1)_\infty}
\check{\Phi}_{\epsilon_1}(z)\check{\Phi}_{\epsilon_2}(w)
\,,
\label{PhiPhi-c}\\
&\Bigl(\frac{w}{z}\Bigr)^{-\check{\beta}}
\frac{(q_1w/z;q_1)_\infty}{(q_1^{1-\check{\beta}}w/z;q_1)_\infty}
\check{\Phi}^*_{\epsilon_1}(z)\check{\Phi}^*_{\epsilon_2}(w)\,,
\label{Phi*Phi*-c}\\
&\check{\Phi}_{\epsilon_1}(z)\check{\Phi}^*_{\epsilon_2}(w)\,,
\quad
\check{\Phi}^*_{\epsilon_1}(z)\check{\Phi}_{\epsilon_2}(w)\,,
\label{PhiPhi*-c}
\end{align} 
have poles only when 
$\epsilon_1+\epsilon_2=0$ and at $z/w=\check{q}_2 q_1^{\check{r}}$, 
$\check{r}\ge0$.
All poles are simple.
\end{lem}
\begin{proof}
The reasoning for \eqref{PhiPhi}, \eqref{Phi*Phi*} and \eqref{PhiPhi*}
are quite similar, so we consider only the last one.

Evidently there are no poles if $(\epsilon_1,\epsilon_2)=(+,+)$. 

Suppose 
$\epsilon_1+\epsilon_2=0$. 
Up to irrelevant constants and powers in $z,w$ we have
\begin{align*}
&\Phi_+(z)\Phi^*_-(w)\propto\int_{\Gamma_+}\frac{dx}{2\pi \sqrt{-1} x}
:\Phi_+(z)\Phi_+^*(w)S(x):
\frac{(s_2q_1x/z;q_1)_\infty \theta_{q_1}(q_1^{P_{1,2}}s_2^{-1}x/w)}
{(s_2^{-1}x/z,s_2^{-1}x/w,s_2^{-1}w/x;q_1)_\infty}\,,
\\
&\Phi_-(z)\Phi^*_+(w)\propto\int_{\Gamma_-}\frac{dx}{2\pi \sqrt{-1} x}
:\Phi_+(z)\Phi_+^*(w)S(x):
\frac{(s_2q_1w/x;q_1)_\infty \theta_{q_1}(q_1^{P_{1,2}}s_2^{-1}x/z)}
{(s_2^{-1}x/z,s_2^{-1}z/x,s_2^{-1}w/x;q_1)_\infty}\,.
\end{align*}
Here the contours are such that 
\begin{align*}
\Gamma_+\quad:&\quad\text{$\{s_2^{-1}q_1^n w\mid n\ge0\}$ are inside and
$\{s_2q_1^{-n}z, s_2q_1^{-n}w\mid n\ge 0\}$ are outside},\\
\Gamma_-\quad:&\quad\text{$\{s_2^{-1}q_1^n z,\ s_2^{-1}q_1^n w
\mid n\ge0\}$ are inside and
$\{s_2q_1^{-n}z\mid n\ge 0\}$ are outside}.
\end{align*}
In both cases, the integral is
holomorphic in $z,w$ unless
 $z/w=q_2^{-1}q_1^r$ with some $r\ge0$ where the contours get pinched.

In the case $(\epsilon_1,\epsilon_2)=(-,-)$ we have
\begin{align*}
&\Phi_-(z)\Phi_-^*(w)\propto
\int_{\Gamma_1}\frac{dx_1}{2\pi \sqrt{-1} x_1} 
\int_{\Gamma_2}\frac{dx_2}{2\pi \sqrt{-1} x_2}
:\Phi_+(z)\Phi^*_+(w)S(x_1)S(x_2): \\
&\times\frac{(s_2q_1x_2/z,s_2q_1w/x_1;q_1)_\infty
\theta_{q_1}(s_2^{-1}q_1^{P_{1,2}+\beta}x_1/z)
\theta_{q_1}(s_2^{-1}q_1^{P_{1,2}}x_2/w)
}
{(s_2^{-1}x_1/z,s_2^{-1}x_2/z,s_2^{-1}w/x_1,s_2^{-1}x_2/w,s_2^{-1}z/x_1,s_2^{-1}w/x_2;q_1)_\infty}
\bigl(1-\frac{x_2}{x_1}\bigr)
\frac{(q_2^{-1}x_2/x_1;q_1)_\infty}{(q_1q_2x_2/x_1;q_1)_\infty}\,.
\end{align*}
We choose $\Gamma_1$ to surround $\{s_2q_1^{-n}z, s_2q_1^{-n}w
\mid n\ge0\}$, and
$\Gamma_2$ to surround $\{s_2^{-1}q_1^nw\mid n\ge0\}$, keeping 
$\{s_2q_1^{-n}z, s_2q_1^{-n}w\mid n\ge0\}$ 
and $\{q_1^{-n-1}q_2^{-1}x_1\mid n\ge0\}$
outside $\Gamma_2$. 
The contours get pinched only when $z/w=q_1^{n}$ with some $n>0$. 
Hence poles of $\Phi_-(z)\Phi_-^*(w)$ are confined to 
$z/w=q_1^n$, $n>0$.

Similarly, $\Phi^*_-(w)\Phi_-(z)$ has poles only at $w/z=q_1^n$ with
$n>0$. 

However, by Proposition \ref{commPhi},
$\Phi_-(z)\Phi^*_-(w)=\Phi^*_-(w)\Phi_-(z)$ holds
as meromorphic functions. 
So these poles are actually absent. 

Statements about the checked version can be seen by noting that the 
kernels of the integrals giving \eqref{PhiPhi-c}, \eqref{Phi*Phi*-c} and 
\eqref{PhiPhi*-c} are obtained from those of 
\eqref{PhiPhi}, \eqref{Phi*Phi*} and \eqref{PhiPhi*} by 
keeping $q_1$ fixed and changing 
$\beta$ to $\check{\beta}$, which amounts to changing $q_2$ to 
$\check{q}_2^{-1}$.
\end{proof}
\medskip

We compute the residues of \eqref{PhiPhi*} and \eqref{PhiPhi*-c}.
Recall the definition of $\tilde{\Phi}_\pm(z)$, $\tilde{\Phi}^*_\pm(z)$
in \eqref{Phitilde2}.
\begin{lem}\label{lem:EFres}
For $r,\check{r}\ge0$,
we have
\begin{align}
&
\mathop{\mathrm{Res}}_{w=q_1^{-1+\beta-r}z}
\tilde{\Phi}_\pm(z)\tilde{\Phi}_{\mp}^*(w)\frac{dw}{w}
=\pm\frac{B^2}{[\beta][P_{1,2}]}\Delta k^-_r(q_1^{-r}z)\,,
\label{ResPP*}\\&\mathop{\mathrm{Res}}_{z=q_1^{-1+\beta-r}w}
\tilde{\Phi}_{\pm}^*(w)\tilde{\Phi}_\mp(z)\frac{dz}{z}
= \pm\frac{B^2}{[\beta][P_{1,2}]}
\Delta k_r^+(q_1^{-r}w)\,,
\label{ResP*P}
\\
&\mathop{\mathrm{Res}}_{{w}= q_1^{1-\check{\beta}+\check{r}}{z}}
\check{\Phi}_\pm^*({w})\check{\Phi}_{\mp}({z})
\frac{d{w}}{{w}}
=\check{B}^2\frac{[\check{\beta}\mp \check{P}_{1,2}]}{[\check{\beta}]}
\Delta \check{k}_{\check{r}}^-(\cq_1^{-\check{r}}{z})
\,,
\label{ResP*Pc}
\\
&
\mathop{\mathrm{Res}}_{{z}= q_1^{1-\check{\beta}+\check{r}}{w}}
\check{\Phi}_{\pm}({z})\check{\Phi}_\mp^*({w})
\frac{d{z}}{{z}}
=\check{B}^2\frac{[\check{\beta}\mp\check{P}_{1,2}]}{[\check{\beta}]}
\Delta \check{k}_{\check{r}}^+(\cq_1^{-\check{r}}{w})
\,.
\label{ResPPc*}
\end{align}
Here $\Delta k^\pm_r(z)$ stands for
 the action of $k^\pm_r(z)$ on $\F_2(v_1)\otimes\F_2(v_2)$, 
and similarly for $\Delta \check{k}^\pm_{\check{r}}(z)$.
\end{lem}
\begin{proof}
The proof being very similar, we show \eqref{ResPP*} 
as an example.
We set $z=q_1^u$, $w=q_1^v$, $x=q_1^t$, and consider
\begin{align}
&\Phi_+(z)\Phi^*_-(w)=\int_{\Gamma_+}
\Phi(z)\Phi^*(x)\otimes\Phi^*(w)\Phi(x)\frac{[t-v+(1-\beta)/2+P_{1,2}]}
{[t-v+(1+\beta)/2]}\frac{dx}{2\pi \sqrt{-1} x}\,,
\label{Ph+Ph*-}\\ 
&\Phi_-(z)\Phi^*_+(w)=\int_{\Gamma_-}
\Phi(z)\Phi^*(x)\otimes\Phi(x)\Phi^*(w)
\frac{[t-u+(1+\beta)/2+P_{1,2}]}
{[t-u+(1+\beta)/2]}\frac{dx}{2\pi \sqrt{-1} x}\,.
\label{Ph-Ph*+} 
\end{align}
Here, for both $\Gamma_+$ and $\Gamma_-$, 
the points $\{s_2^{-1}q_1^nw\mid n\ge0\}$ are inside 
and $\{s_2q_1^{-n}z\mid n\ge0\}$ are outside. In addition,
$\{s_2q_1^{-n}w\mid n\ge0\}$ are outside $\Gamma_+$ and 
$\{s_2^{-1}q_1^nz\mid n\ge0\}$ are inside $\Gamma_-$.

As $w\to q_2q_1^{-r}z$, the contours get pinched by the 
poles $\{s_2^{-1}q_1^n w\mid 0\le n\le r\}$
and $\{s_2q_1^{-n}z\mid 0\le n\le r\}$. 
Hence the leading terms of the integrals are given by
\begin{align*}
&\Phi_+(z)\Phi_-^*(w)=
-\sum_{l=0}^r \mathop{\mathrm{Res}}_{x=s_2q_1^{-l}z}
\Phi(z)\Phi^*(x)\frac{dx}{x}
\otimes\Phi^*(w)\Phi(s_2q_1^{-l}z)
\,\frac{[u-v+P_{1,2}]}{[u-v+\beta]} +O(1)\,, 
\\
&\Phi_-(z)\Phi_+^*(w)=
-\sum_{l=0}^r \mathop{\mathrm{Res}}_{x=s_2q_1^{-l}z}
\Phi(z)\Phi^*(x)\frac{dx}{x}
\otimes \Phi(s_2q_1^{-l}z)\Phi^*(w)
\,\frac{[P_{1,2}+\beta]}{[\beta]} +O(1)\,. 
\end{align*}
Here we use $[u+1]=-[u]$. 

Applying \eqref{PPK1}, \eqref{PPK3} in Lemma \ref{lem:PhiK},
and using the coproduct of $k^-_r(z)$ given in Lemma \ref{lem:DeltaK},
we obtain
\begin{align*}
\mathop{\mathrm{Res}}_{w=q_2q_1^{-r}z}		      
\Phi_+(z)\Phi_-^*(w)\frac{dw}{w}
&=\sum_{l=0}^rB k_l^-(q_1^{-l}z)\otimes Ak_{r-l}^-(s_2q_1^{-r})
\mathop{\mathrm{Res}}_{w=q_2q_1^{-r}z}\frac{[u-v+P_{1,2}]}{[u-v+\beta]}
\frac{dw}{w}\\
&=B^2\frac{[P_{1,2}-\beta]}{[\beta]}\Delta k_r^-(q_1^{-r}z)\,,
\\
\mathop{\mathrm{Res}}_{w=q_2q_1^{-r}z}
\Phi_-(z)\Phi_+^*(w)\frac{dw}{w}
&=-\sum_{l=0}^rBk_l^-(q_1^{-l}z)\otimes Bk_{r-l}^-(s_2q_1^{-r})
\frac{[P_{1,2}+\beta]}{[\beta]}
\\
&=-B^2\frac{[P_{1,2}+\beta]}{[\beta]}\Delta k_r^-(q_1^{-r}z)\,.
\end{align*}
Noting further that 
\begin{align*}
\tilde{\Phi}_{\pm}(z)\tilde{\Phi}^*_{\mp}(w)
={\Phi}_{\pm}(z){\Phi}^*_{\mp}(w)\frac{1}{[P_{1,2}][P_{1,2}\mp\beta]}
\end{align*}
we arrive at \eqref{ResPP*}.
\end{proof}
\medskip

\begin{lem}\label{lem:EFvalue}
For $r,\check{r}\ge0$ we have
\begin{align}
&\Phi_+(q_2q_1^{-r}w)\Phi_-^*(w) 
-\Phi_-(q_2q_1^{-r}w)\Phi_+^*(w) 
=A B \frac{[P_{1,2}]}{[\beta]}
\Delta k_r^+(q_1^{-r}w)\,,
\label{valPP}\\
&\check{\Phi}_+(z)\check{\Phi}_-^*(\cq_2\cq_1^{-\check{r}}z) 
-\check{\Phi}_-(z)\check{\Phi}_+^*(\cq_2\cq_1^{-\check{r}}z) 
=\check{A}\check{B} \frac{[\check{P}_{1,2}]}{[\check{\beta}]}
\Delta k_{\check{r}}^-(\cq_1^{-\check{r}}z)\,.
\label{valcPP}
\end{align}
\end{lem}
\begin{proof}
The two cases being similar, we show only \eqref{valPP}. 
We retain the notation in the proof of Lemma \ref{lem:EFres}.
Specializing \eqref{Ph+Ph*-} 
to $z=q_2q_1^{-r}w$, we obtain
\begin{align}
&\Phi_+(q_2q_1^{-r}w)\Phi^*_-(w)=
\int_{\Gamma_+} \frac{dx}{2\pi\sqrt{-1}x}
:\Phi(z)\Phi^*(x)\otimes\Phi^*(w)\Phi(x):\label{PP*special}\\
&\times
(q_2q_1^{-r}w^2)^{-\beta}
\frac{1}{(s_2^{-1}x/w;q_1)_{r+1}}
\frac{(s_2q_1 x/w;q_1)_\infty}{(s_2^{-3}q_1^rx/w;q_1)_\infty}
\frac{[t-v+(1-\beta)/2+P_{1,2}]}{[t-v+(1+\beta)/2]}\,.
\nn
\end{align}
On the other hand, by using \eqref{PhiPhi*1}, \eqref{Ph-Ph*+}  
can be rewritten as
\begin{align*}
&\Phi_-(z)\Phi^*_+(w)=\int_{\Gamma_-}
\Phi(z)\Phi^*(x)\otimes\Phi^*(w)\Phi(x)
\frac{[t-v+(1-\beta)/2]}{[t-v+(1+\beta)/2]}
\frac{[t-u+(1+\beta)/2+P_{1,2}]}
{[t-u+(1+\beta)/2]}\frac{dx}{2\pi \sqrt{-1} x}\,.
\end{align*}
Upon specialization to $z=q_2q_1^{-r}w$, the integrand becomes the same
as that of \eqref{Ph+Ph*-}.

It follows that
\begin{align*}
&\Phi_+(q_2q_1^{-r}w)\Phi^*_-(w)-\Phi_-(q_2q_1^{-r}w)\Phi^*_+(w)
=-\sum_{l=0}^r \mathop{\mathrm{Res}}_{x=s_2q_1^{-l}w}\frac{dx}{x}
:\Phi(q_2q_1^{-r}w)\Phi^*(x)\otimes\Phi^*(w)\Phi(x):
\\
&\times(q_2q_1^{-r}w^2)^{-\beta}
\frac{1}{(s_2^{-1}x/w;q_1)_{r+1}}
\frac{(s_2q_1 x/w;q_1)_\infty}{(s_2^{-3}q_1^rx/w;q_1)_\infty}
\frac{[t-v+(1-\beta)/2+P_{1,2}]}{[t-v+(1+\beta)/2]}\,.
\end{align*}
We compute the right hand side substituting
\begin{align*}
:\Phi(s_2z)\Phi^*(q_1^rz):=s_2^{\beta/2}z^{\beta}
\frac{(q_1;q_1)_r}{(q_2^{-1};q_1)_r}k^+_r(z)\,,
\end{align*}
which follows from \eqref{PPK1}, to obtain
\begin{align*}
&
\frac{(q_2q_1;q_1)_\infty} {(q_2^{-1};q_1)_\infty} 
\frac{[P_{1,2}]}{[\beta]}
\sum_{l=0}^rk^+_{r-l}(s_2q_1^{-r}w)\otimes k_l^+(q_1^{-l}w)
\\
&=AB \frac{[P_{1,2}]}{[\beta]}
\Delta k^+_r(q_1^{-r}w)\,.
\end{align*}
\end{proof}

\subsection{Proof}\quad
Now we move to the proof of Theorem \ref{thm:level2}. 

\begin{lem}\label{lem:rel-rat}
We have identities as rational functions
\begin{align}
&
\gamma_{i,j}(z,w)X^+_i(z)X^+_j(w) 
=
(-1)^{i-j-1}\gamma_{j,i}(w,z)X^+_j(w)X^+_i(z)\,,
\label{EE2222}
\\
&
\gamma_{i,j}(w,z)
X^-_i(z)X^-_j(w)
=
(-1)^{i-j-1}\gamma_{j,i}(w,z)X^-_j(w)X^-_i(z)\,,
\label{FF222}
\\
&X^+_i(z)X^-_j(w)=X^-_j(w)X^+_i(z)\,.
\label{EF2222}
\end{align} 
\end{lem}
\begin{proof}
We write
$R(z)=R(u;\beta,P_{1,2})$,  
$\check{R}(z)=R(u;\check{\beta},\check{P}_{1,2})$ 
for $z=q_1^u$.
We have $\check{\beta}\equiv -\beta$ $\bmod\, \Z$, 
and 
the eigenvalues of $P_{1,2}+\check{P}_{1,2}$ are integers
by \eqref{P12P12}.
Since $R(u;\beta,P)$ is separately $\Z$-periodic in all its 
arguments, we obtain 
\begin{align*}
R(u;\check{\beta},\check{P}_{1,2})
&=R(u;-\beta,-P_{1,2})\\
&=R(-u;\beta,P_{1,2})\\
&=R(u;\beta,P_{1,2})^{-1}\,.
\end{align*}
This means $R(z/w)\check{R}(z/w)=\id$.

Hence the commutation relations in Proposition \ref{commPhi} imply that
\begin{align*}
X^+_i(z)X^-_j(w)
&=\mathcal{N}_2
\sum_{\epsilon_1,\epsilon_2}
\tilde{\Phi}_{\epsilon_1}(z) \tilde{\Phi}^*_{\epsilon_2}(w) 
\check{\Phi}_{\epsilon_1}(q_1^iz) 
\check{\Phi}^{*}_{\epsilon_2}(q_1^{-j}w) \\
&=\mathcal{N}_2
\sum_{\epsilon'_1,\epsilon'_2, \epsilon''_1,\epsilon''_2}
\tilde{\Phi}^*_{\epsilon'_1}(w) \tilde{\Phi}_{\epsilon'_2}(z) 
\Bigl({}^tR(z/w){}^t\check{R}(q_1^{i+j}z/w)
\Bigr)_{\epsilon_1'\epsilon_2';\epsilon_1''\epsilon_2''}
\check{\Phi}^{*}_{\epsilon''_1}(q_1^{-j}w) 
\check{\Phi}_{\epsilon''_2}(q_1^iz) \,
\\
&=X^-_j(w)X^+_i(z)\,.
\end{align*} 
In particular, this implies that both sides are rational.

The other relations hold for the same reason,
in view of the relations
\begin{align*}
\rho(z/w)\check{\rho}(q_1^{i-j}z/w) 
&=(-1)^{i-j-1}\frac{\gamma_{j,i}(w,z)}{\gamma_{i,j}(z,w)}
\,,
\\
\rho^*(z/w)\check{\rho}^*(q_1^{-i+j}z/w) 
&=(-1)^{i-j-1}\frac{\gamma_{j,i}(z,w)}{\gamma_{i,j}(w,z)}\,.
\end{align*}
\end{proof}

\begin{lem}\label{lem:EE-FF}
The $X^+X^+$ and $X^-X^-$ relations (R2) hold. 
\end{lem}
\begin{proof}
In view of Lemma \ref{lem:rel-rat},
they are satisfied as rational functions. 

On the other hand, by Lemma \ref{lem:location},
the poles of 
\begin{align*}
z^{-M-1}\frac{(q_1w/z;q_1)_{j-i-1}}{(q_2^{-1}w/z;q_1)_{j-i+M}} 
X^+_i(z)X^+_j(w) 
\end{align*}
are at
\begin{align*}
\frac{w}{z}=q_2q_1^{-r}\,,\quad r\ge 0\,;
\quad \frac{w}{z}=q_2^{-1}q_1^r\,,\quad r\le M+i-j-1\,,
\end{align*}
and the poles of 
\begin{align*}
w^{-M-1}\frac{(q_1z/w;q_1)_{i-j-1}}{(q_2^{-1}z/w;q_1)_{i-j+M}}  
X^+_j(w)X^+_i(z)
\end{align*}
are at
\begin{align*}
\frac{w}{z}=q_2q_1^{-r}\,,\quad r\le M-i+j-1\,;
\quad \frac{w}{z}=q_2^{-1}q_1^r\,,\quad r\ge 0\,.
\end{align*}
We may assume $M\ge0$, $i\le j$. 
Comparing with the definition of $\gamma_{i,j}(z,w)$, 
we find that 
$\gamma_{i,j}(z,w)X^+_i(z)X^+_j(w)$ has poles at most at the following:
\begin{align*}
i=j\quad:&\quad  \frac{w}{z}=q_2q_1^{-r}\,,\quad r\ge M;
\quad \frac{w}{z}=q_2^{-1}q_1^{r}\,,\quad r\le -1\,,
\\
0<j-i\le M\quad:&\quad  
\frac{w}{z}=q_2q_1^{-r}\,,\quad r\ge M+j-i;
\quad \frac{w}{z}=q_2^{-1}q_1^{r}\,,\quad r\le -1\,,
\\
j-i>M\quad:&\quad  
\frac{w}{z}=q_2q_1^{-r}\,,\quad r\ge M+j-i;
\quad \frac{w}{z}=q_2^{-1}q_1^{r}\,,\quad r\le M-j+i-1\,.
\end{align*}
Similarly 
$\gamma_{j,i}(w,z)X^+_j(w)X^+_i(z)$ has poles at most at the following:
\begin{align*}
i=j\quad:&\quad  \frac{w}{z}=q_2q_1^{-r}\,,\quad r\le -1;
\quad \frac{w}{z}=q_2^{-1}q_1^{r}\,,\quad r\ge M\,,
\\
0<j-i\le M\quad:&\quad  
\frac{w}{z}=q_2q_1^{-r}\,,\quad r\le -1;
\quad \frac{w}{z}=q_2^{-1}q_1^{r}\,,\quad r\ge M-j+i\,,
\\
j-i>M\quad:&\quad  
\frac{w}{z}=q_2q_1^{-r}\,,\quad r\le -1;
\quad \frac{w}{z}=q_2^{-1}q_1^{r}\,,\quad r\ge0\,.
\end{align*}
In each case these sets do not intersect.
It follows that
both sides of the relation \eqref{EE2222} are actually Laurent polynomials. 
Hence the $X^+X^+$ relations hold as formal series relations. 

Proof of 
the $X^-X^-$ relations is entirely similar.
\end{proof}

\begin{lem}\label{lem:EFpole}
Operator $X^+_i(z)X^-_j(w)$ has at most simple poles at 
\begin{align}
&q_1^{1-\beta+r}\frac{w}{z}=1,\quad 0\le r\le -i-j+M-1\,,
\label{poleEF-}
\\
&q_1^{1-\beta+r}\frac{z}{w}=1\,,\quad 0\le r\le i+j+M-1\,.
\label{poleEF+}
\end{align}
\end{lem}
\begin{proof}
This can be shown by a similar argument based on
Lemma \ref{lem:location}.

\end{proof}

It remains to compute the residues of $X^+_i(z)X^-_j(w)$ at these points. 
\bigskip

\noindent{\it End of Proof of Theorem \ref{thm:level2}.}\quad
As noted before, if we write $\check{z}=q_1^iz$ and $\check{w}=q_1^{-j}w$, then
\eqref{poleEF-}, \eqref{poleEF+} 
can be rewritten respectively as
\begin{align*}
&
q_1^{1-\check{\beta}+\check{r}}
\frac{\check{z}}{\check{w}}=1\,,\quad r+\check{r}=-i-j+M-1\,,
\quad r,\check{r}\ge0\,,
\\
&
q_1^{1-\check{\beta}+\check{r}}
\frac{\check{w}}{\check{z}}=1
\,,\quad r+\check{r}=i+j+M-1\,,
\quad r,\check{r}\ge0\,.
\end{align*}
Using \eqref{ResP*P} and \eqref{valcPP} we obtain
\begin{align*}
\mathcal{N}_2^{-1}
\mathop{\mathrm{Res}}_{w=q_2q_1^{-r}z}X_i^+(z)X_j^-(w)\frac{dw}{w} 
&=\mathop{\mathrm{Res}}_{w=q_2 q_1^{-r} z}\sum_{\epsilon=\pm}
\tilde{\Phi}_\epsilon(z)\tilde{\Phi}^*_{-\epsilon}(w)
\check{\Phi}_\epsilon(\check{z})
\check{\Phi}^*_{-\epsilon}(\check{w})
\frac{dw}{w}
\\
&=\frac{B^2}{[\beta][P_{1,2}]}\Delta k^-_r(q_1^{-r}z)
\Bigl\{
\check{\Phi}_+(\check{z})\check{\Phi}^*_-(\cq_2\cq_1^{-\check{r}}\check{z})
-
\check{\Phi}_-(\check{z})\check{\Phi}^*_+(\cq_2\cq_1^{-\check{r}}\check{z})
\Bigr\}
\\
&=\frac{\check{A}\check{B}B^2}{[\beta][\check\beta]}
\frac{[\check{P}_{1,2}]}{[P_{1,2}]}
\Delta k^-_r(q_1^{-r}z)
\Delta \check{k}^-_{\check{r}}(\cq_1^{-\check{r}}\check{z})\,
\\
&=(-1)^{L+1}\frac{\check{A}\check{B}B^2}{[\beta][\check\beta]}
\Delta k^-_r(q_1^{-r}z)
\Delta \check{k}^-_{\check{r}}(\cq_1^{-\check{r}}\check{z})\,.
\end{align*}
In the last line we use the fact 
that $P_{1,2}+\check{P}_{1,2}\equiv L\bmod 2\Z$, see \eqref{P12P12}.

Next we use \eqref{ResPPc*} to compute
\begin{align*}
\mathcal{N}_2^{-1}
\mathop{\mathrm{Res}}_{z=q_2q_1^{-r}w}X_i^+(z)X_j^-(w)\frac{dz}{z} 
&=\mathop{\mathrm{Res}}_{z=q_2 q_1^{-r} w}\sum_{\epsilon=\pm}
\tilde{\Phi}_\epsilon(z)\tilde{\Phi}^*_{-\epsilon}(w)
\check{\Phi}_\epsilon(\check{z})
\check{\Phi}^*_{-\epsilon}(\check{w})
\frac{d\check{z}}{\check{z}}
\\
&=\sum_{\epsilon=\pm}
\tilde{\Phi}_\epsilon(q_2q_1^{-r}w)\tilde{\Phi}_{-\epsilon}^*(w)
\frac{[\check{\beta}-\epsilon\check{P}_{1,2}]}{[\check{\beta}]}
\check{B}^2\Delta \check{k}^+_{\check{r}}(\cq_1^{-\check{r}}\check{w})
\\
&=\sum_{\epsilon=\pm}
\Phi_\epsilon(q_2q_1^{-r}w)\Phi_{-\epsilon}^*(w)
\frac{-1}{[P_{1,2}-\epsilon\beta][P_{1,2}]}
\frac{[\check{\beta}-\epsilon\check{P}_{1,2}]}{[\check{\beta}]}
\check{B}^2\Delta \check{k}^+_{\check{r}}(\cq_1^{-\check{r}}\check{w})
\,.
\end{align*}
We have $\check\beta\equiv -\beta\bmod \Z$, so that
\begin{align*}
\frac{[\check{\beta}-\epsilon\check{P}_{1,2}]}{[\check{\beta}]}
=
\frac{[\beta+\epsilon\check{P}_{1,2}]}{[\beta]} 
=(-1)^L
\frac{[\beta-\epsilon {P}_{1,2}]}{[\beta]}\,. 
\end{align*}
Hence we obtain from \eqref{valPP}
\begin{align*}
\mathcal{N}_2^{-1}
 \mathop{\mathrm{Res}}_{z=q_2q_1^{-r}w}X_i^+(z)X_j^-(w)\frac{dz}{z} 
&=
\Bigl\{
\Phi_+(q_2q_1^{-r}w)\Phi_{-}^*(w)
-
\Phi_{-}(q_2q_1^{-r}w)\Phi_{+}^*(w)
\Bigr\}
(-1)^L\frac{\check{B}^2}{[P_{1,2}][\beta]}
\Delta \check{k}^+_{\check{r}}(\cq_1^{-\check{r}}\check{w})
\\
&=(-1)^L \frac{AB\check{B}^2}{[\beta]^2}
\Delta k^+_r(q_1^{-r}w)
\Delta \check{k}^+_{\check{r}}(\cq_1^{-\check{r}}\check{w})\,.
\end{align*}
We choose
\begin{align*}
\mathcal{N}_2^{-1}= (-1)^L \frac{AB\check{B}^2}{[\beta]^2}\,.
\end{align*}
Noting the identity
\begin{align*}
\frac{\check{A}B}{A\check{B}}=\frac{[\beta]}{[\check{\beta}]}
\end{align*}
we conclude that
\begin{align*}
 \mathop{\mathrm{Res}}_{z=q_2q_1^{-r}w}X_i^+(z)X_j^-(w)\frac{dz}{z} 
&=
\Delta k^+_r(q_1^{-r}w)
\Delta \check{k}^+_{\check{r}}(\cq_1^{-\check{r}}\check{w})\,, 
\\
\mathop{\mathrm{Res}}_{w=q_2q_1^{-r}z}X_i^+(z)X_j^-(w)\frac{dw}{w} 
&=-
\jm{\Delta k^-_r(q_1^{-r}z)
\Delta \check{k}^-_{\check{r}}(\cq_1^{-\check{r}}\check{z})}\,.
\end{align*}

The proof is over.

\qed
\medskip

\section{Representation $\mathbb{F}_{2,\ldots,2;2,\ldots,2}$}
\label{sec:2222}

In this section we sketch the construction of the representation
$\mathbb{F}_{2,\ldots,2;2,\ldots,2}(v_1,\ldots,v_n;
\check{v}_1,\ldots,\check{v}_n)$ for odd $M$.
As before, we write 
$v_i=q_1^{\lambda_i}$, 
$\check{v}_i=\cq_1^{\check{\lambda}_i}$ and assume that 
$\lambda_i+\check{\lambda}_i\in\Z$, $1\le i\le n$.
\medskip

For $1\le i\le n$ we set
$P_i=\id^{\otimes(i-1)}\otimes P\otimes \id^{\otimes(n-i)}$, 
$P_{i,j}=P_i-P_j$, 
and
\begin{align*}
S^{(i)}(x)&=\id^{\otimes (i-1)}\otimes\Phi^*(x)\otimes\Phi(x)\otimes\id^{\otimes(n-i-1)}\,.
\end{align*}
Define
\begin{align*}
\Phi^{(1)}(z)&=\Phi(z)\otimes {\id}^{\otimes(n-1)}\,,\quad
\Phi^{*(n)}(w)={\id}^{\otimes(n-1)}\otimes \Phi^*(w)\,,
\end{align*}
and 
\begin{align}
\Phi^{(i)}(z)&=\int\!\cdots\!\int\prod_{j=1}^{i-1}\frac{dx_j}{2\pi \sqrt{-1} x_j}
\Phi^{(1)}(z)S^{(1)}(x_1)\cdots S^{(i-1)}(x_{i-1})
\prod_{j=1}^{i-1}
\frac{[t_j-t_{j-1}-P_{i,j}+\frac{1-\beta}{2}]}
{[t_j-t_{j-1}+\frac{1+\beta}{2}]}\,,
\label{Phii}\\
\Phi^{*(i)}(w)&=\int\!\cdots\!\int\prod_{j=i}^{n-1}\frac{dx_j}{2\pi \sqrt{-1} x_j}
S^{(i)}(x_{i})\cdots S^{(n-1)}(x_{n-1})
\Phi^{*(n)}(w)
 \prod_{j=i+1}^{n}
\frac{[t_j-t_{j-1}-P_{i,j}+\frac{1+\beta}{2}]}
{[t_j-t_{j-1}+\frac{1+\beta}{2}]}\,,
\label{Phii*}
\end{align}
where  $x_0=z$, $x_n=w$, $x_i=q_1^{t_i}$.
In \eqref{Phii}, for $1\le j\le i-1$, the contour for $x_j$ separates the set
$\{s_2^{-1}q_1^lx_{j-1}\mid l\ge 0\}$ and 
$\{s_2q_1^{-l}x_{j-1}\mid l\ge 0\}$.
In \eqref{Phii*}, for $i\le j\le n-1$, the contour for $x_j$ separates the set
$\{s_2^{-1}q_1^lx_{j+1}\mid l\ge 0\}$ and
$\{s_2q_1^{-l}x_{j+1}\mid l\ge 0\}$.

For $n=2$, the notation in Section \ref{sec:s22s22}
is related to the one above
as $\Phi_+(z)=\Phi^{(1)}(z)$, $\Phi_-(z)=\Phi^{(2)}(z)$,
 $\Phi^*_+(z)=\Phi^{*(2)}(z)$, $\Phi^*_-(z)=\Phi^{*(1)}(z)$.
\medskip

The non-trivial commutation relations among these operators 
are as follows:
\begin{align*}
&S^{(1)}(x)\Phi^{(1)}(z)=
\Phi^{(1)}(z)S^{(1)}(x)
\frac{[t-u+\frac{1-\beta}{2}]}
{[t-u+\frac{1+\beta}{2}]}\,,
\\
&S^{(n-1)}(x)\Phi^{*(n)}(w)=
\Phi^{*(n)}(w)S^{(n-1)}(x)
\frac{[t-v+\frac{1-\beta}{2}]}
{[t-v+\frac{1+\beta}{2}]}\,,
\\
&S^{(i)}(x_1)S^{(i)}(x_2)=
S^{(i)}(x_2)S^{(i)}(x_1)
\frac{[t_1-t_2+\beta]}{[t_1-t_2-\beta]}\,,
\\
&S^{(i)}(x_1)S^{(i+1)}(x_2)=S^{(i+1)}(x_2)S^{(i)}(x_1)
\frac{[t_1-t_2+\frac{1-\beta}{2}]}{[t_1-t_2+\frac{1+\beta}{2}]}\,.
\end{align*}

The following is a counterpart of the result of \cite{AJMP}
where the deformed $W_n$ algebra was treated.

\begin{thm}\label{thm:AJMP}
We set  $z_i=q_1^{u_i}$. Then
the following commutation relations hold: 
\begin{align*}
&\Phi^{(i)}(z_1) \Phi^{(i)}(z_2) 
=\rho(z_1/z_2)\Phi^{(i)}(z_2) \Phi^{(i)}(z_1) \,,\\
&\Phi^{(i)}(z_1) \Phi^{(j)}(z_2) 
=\rho(z_1/z_2)\\
&\times\Bigl(
\Phi^{(j)}(z_2) \Phi^{(i)}(z_1)\frac{[u_1-u_2]}{[u_1-u_2+\beta]}
\frac{[P_{i,j}-\beta]}{[P_{i,j}]}
+\Phi^{(i)}(z_2) \Phi^{(j)}(z_1)
\frac{[u_1-u_2+P_{i,j}]}{[u_1-u_2+\beta]}\frac{[\beta]}{[P_{i,j}]}
\Bigr) 
\quad \text{for $i\neq j$}, 
\\
&\Phi^{*(i)}(z_1) \Phi^{*(i)}(z_2) 
=\rho^*(z_1/z_2)\Phi^{*(i)}(z_2) \Phi^{*(i)}(z_1) \,,\\
&\Phi^{*(i)}(z_1) \Phi^{*(j)}(z_2) 
=\rho^*(z_1/z_2)\\
&\times\Bigl(
\Phi^{*(j)}(z_2) \Phi^{*(i)}(z_1)\frac{[u_1-u_2]}{[u_1-u_2+\beta]}
\frac{[P_{j,i}-\beta]}{[P_{j,i}]}
+\Phi^{*(i)}(z_2) \Phi^{*(j)}(z_1)
\frac{[u_1-u_2+P_{j,i}]}{[u_1-u_2+\beta]}\frac{[\beta]}{[P_{j,i}]}
\Bigr) 
\quad \text{for $i\neq j$},
\\
&\Phi^{(i)}(z_1)\Phi^{*(j)}(z_2)=\Phi^{*(j)}(z_2)\Phi^{(i)}(z_1)
\quad \text{for $i\neq j$},
\\
&\Phi^{(i)}(z_1)\Phi^{*(i)}(z_2)
=\sum_{k=1}^n\Phi^{*(k)}(z_2)\Phi^{(k)}(z_1) 
\frac{[u_1-u_2-n\frac{1-\beta}{2}-\beta+P_{i,k}]}{[u_1-u_2-n\frac{1-\beta}{2}]}
\prod_{j(\neq k)}\frac{[P_{j,i}+\beta]}{[P_{j,k}]}\,.
\end{align*}
\end{thm}
\begin{proof}
This can be shown by an adaptation of the method of \cite{AJMP}. 
\end{proof}
\medskip

The $\check{\E}_1$ counterpart 
$\check{\Phi}^{(i)}(z)$, $\check{\Phi}^{*(i)}(z)$, 
$\check{S}^{(i)}(x)$ are defined by replacing $\Phi(z),\Phi^*(z),\beta$
by $\check{\Phi}(z),\check{\Phi}^*(z),\check{\beta}$, respectively.
Commutation relations for them
are obtained from Theorem \ref{thm:AJMP}
replacing $\beta$ by $\check{\beta}$, $P_i$ by $\check{P}_i$, and 
$\rho(z),\rho^*(z)$ by 
$\check{\rho}(z),
\check{\rho}^*(z)$, respectively.

In addition we change the normalization 
\begin{align*}
&\tilde{\Phi}^{(i)}(z)=\Phi^{(i)}(z)\prod_{k(\neq i)}[P_{i,k}]^{-1}\,,
\quad
\tilde{\Phi}^{*(i)}(z)=\Phi^{*(i)}(z)\prod_{k(\neq i)}[P_{i,k}]^{-1}\,.
\end{align*}

\begin{thm}\label{thm:leveln}
With an appropriate choice of the normalization constant
$\mathcal{N}_n$, 
the assignment
\begin{align*}
&X^+_i(z)\mapsto
\sum_{\mu=1}^n \tilde{\Phi}^{(\mu)}(z)
\check{\Phi}^{(\mu)}(q_1^iz)
\,,
\\
&X^-_i(z)\mapsto
\mathcal{N}_n
\sum_{\mu=1}^n 
\tilde{\Phi}^{*(\mu)}(z)
\check{\Phi}^{*(\mu)}(q_1^{-i}z)
\,,
\end{align*}
yields a representation of 
$\mathcal{A}_{M,n(M-1)}$ on 
$\mathbb{F}_{2,\ldots,2;2,\ldots,2}
(v_1,\ldots,v_n;\check{v}_1,\ldots,\check{v}_n)$.
\qed
\end{thm}

\bigskip

%{\bf Data availability statement.}
%The authors declare that the data supporting the findings of this study are available within the paper.

%\medskip

%{\bf Conflict of interest statement.}
%The authors have no competing interests to declare that are relevant to the content of this article.

\medskip

\vskip 1cm

{\bf Acknowledgments.\ }
BF is partially supported by ISF 3037 2025 1 1.
MJ is partially supported by JSPS KAKENHI Grant Numbers 25K07041, 23K03137, 
25K06942. 
EM is partially supported by Simons Foundation grant number \#709444.

The authors thank M. Bershtein for helping with the literature.
BF and EM thank Rikkyo University, where a part of this work was done, 
for hospitality.


\vskip 1cm

\begin{thebibliography}{0000000}

\bibitem[AFS]{AFS} H. Awata, B. Feigin and J. Shiraishi,
{\it Quantum algebraic approach to refined topological vertex},
JHEP {\bf 03} (2012) 041

\bibitem[AJMP]{AJMP} Y. Asai, M. Jimbo, T. Miwa and Y. Pugai,
{\it Bosonization of vertex operators for the $A^{(1)}_{n-1}$ face model},
J. Phys. A: Math. Gen. {\bf 29} (1996) 6595--6616


\bibitem[B]{B} D. Butson, 
{\it Vertex algebras from divisors on Calabi-Yau threefolds},
arXiv:2312.03648[math.RT]
%https://arxiv.org/abs/2312.03648


\bibitem[BBFLT]{BBFLT} A. Belavin, M. Bershtein, B. Feigin, A. Litvinov, 
and G. Tarnopolsky,
{\it Instanton moduli spaces and bases in coset conformal field theory},
Commun. Math. Phys. {\bf 319} (2013) 269--301
%https://arxiv.org/abs/1111.2803

\bibitem[BFL]{BFL} M. Bershtein, B. Feigin, and A. Litvinov, 
Coupling of two conformal theories and Nakajima-Yoshioka blow-up equations,
Lett. Math. Phys. {\bf 106}(2016) 29--56
%https://arxiv.org/abs/1310.7281


\bibitem[CG]{CG} T. Creutzig and D. Gaiotto,
{\it Vertex algebras for S-duality},
Commun. Math. Phys. {\bf 378} (2020) 785--845
%https://arxiv.org/pdf/1708.00875

\bibitem[FHHSY]{FHHSY}
B.Feigin, K.Hashizume, A.Hoshino, J.Shiraishi, and S.Yanagida,
\textit{A commutative algebra on degenerate $\C\mathbb{P}^1$
and Macdonald polynomials}, 
J. Math. Phys. {\bf 50} (2009) 095215
\bibitem[FJM]{FJM} B. Feigin, M. Jimbo, and E. Mukhin,
{\it Extensions of deformed $W$-algebras via $qq$-characters},
Transformation Groups (2024)
https://doi.org/10.1007/s00031-024-09869-w

\bibitem[FJM2]{FJM2} B. Feigin, M. Jimbo, and E. Mukhin,
{\it Affinization of shifted quantum affine $\gl_2$},
arXiv:2511.12178

\bibitem[FJM3]{FJM3} B. Feigin, M. Jimbo, and E. Mukhin,
{\it Evaluation modules for quantum toroidal $\gl_n$ algebras},
Interactions of Quantum Affine Algebras with Cluster Algebras, Current 
Algebras and Categorification, 
Progress in Mathematics {\bf 337}, (2021) 393--425, 
Birkh{\"a}user

\bibitem[FJM4]{FJM4} B. Feigin, M. Jimbo, and E. Mukhin,
{\it Towards trigonometric deformation of $\widehat{\mathfrak{sl}}_2$ coset
  VOA}, 
J. Math. Phys. \textbf{60} (2019) 073507


\bibitem[FOS]{FOS} M. Fukuda, Y. Ohkubo and J. Shiraishi,
{\it Generalized Macdonald functions on Fock tensor spaces and duality formula
for changing preferred direction},
Commun. Math. Phys. {\bf 380} (2020) 1--70

\bibitem[FPT]{FPT} R. Frassek, V. Pestun and A. Tsymbaliuk,
{\it Lax matrices from anitdominantly shifted Yangians and
quantum affine algebras: A-type}, 
Advances in Math. {\bf 401} (2022) 108283

\bibitem[FT]{FT} M. Finkelberg and A. Tsymbaliuk,
{\it Multiplicative slices, relativistic Toda and shifted quantum affine 
algebras}, in 
Representations and Nilpotent Orbits of Lie Algebraic Systems, 
pp. 133--304, Birkh{\"a}user, 2019.

\bibitem[La]{La} M. Lashkevich, 
{\it New conformal models with $c<2/5$}, 
Mod. Phys. Lett. {\bf A9} (1994) 2273-2280

\bibitem[L]{L} S. Lukyanov,
{\it Free field representation for massive integrable models},
Commun. Math. Phys. {\bf 167} (1995) 183--226

\bibitem[LL]{LL} W. Li and P. Longhi,
{\it Gluing two affine Yangians of $\gl_1$},
JHEP {\bf 10} (2019) 131
%https://arxiv.org/abs/1905.03076


\bibitem[LP]{LP} S. Lukyanov and Y. Pugai,
{\it Multi-point local height probabilities in the integrable RSOS model},
Nucl. Phys. B {\bf 473} [FS 63] (1996) 631--658


\bibitem[Mi]{Mi} K. Miki,  {\it Toroidal and level $0$ $U_q'\widehat{sl_{n+1}}$ actions on $U_q\widehat{gl_{n+1}}$
modules}, J. Math. Phys., {\bf 40} (1999), no. 6, 3191--3210


\bibitem[N]{N} A. Negut, 
{\it Toward AGT for parabolic sheaves}, 
%arXiv:1911.02963
Int. Math. Res. Notices {\bf 2022} 6512--6539

\bibitem[NW]{NW} G. Noshita and A. Watanabe, 
{\it Shifted quiver quantum toroidal algebra and subcrystal representations},
JHEP article no.122, (2022)

\bibitem[PR]{PR} T. Proch{\'{a}}zka and M. Rap{\v{c}}{\'{a}}k,
{\it Webs of W-algebras},
JHEP {\bf 2018}  no. 109
%https://arxiv.org/abs/1711.06888 

\bibitem[Sh]{Sh} J. Shiraishi, 
{\it Free boson representation of $U_q(\widehat{sl}_2)$}, 
Phys. Lett. A {\bf 171} (1992) 243--248

\bibitem[Z]{Z} Y. Zenkevich, 
{\it Higgsed network calculus}, 
JHEP {\bf 08},  (2021) 149


\end{thebibliography}
\end{document}